\documentclass[10pt,a4paper,twoside]{amsart}

\usepackage{amssymb}
\usepackage{amsmath,amssymb,epsf,wrapfig,epic,eepic,trig,mathrsfs,dsfont,color,graphicx,accents}
 \usepackage[abs]{overpic}

\theoremstyle{plain}
\newtheorem{thm}{Theorem}[section]
\newtheorem{prop}[thm]{Proposition}
\newtheorem{cor}[thm]{Corollary}
\newtheorem{lem}[thm]{Lemma}

\theoremstyle{remark}
\newtheorem{rmk}[thm]{Remark}

\theoremstyle{definition}
\newtheorem{defn}[thm]{Definition}

\def\rmoveio#1#2{%#2=2 oriented #2=1 non-oriented
\setlength{\unitlength}{#1}
\begin{picture}(50,30)
\put(5,0){\line(0,1){30}}

{\allinethickness{.8pt}
\put(10,15){\vector(1,0){15}}
\put(25,15){\vector(-1,0){15}}}

\qbezier(30,0)(30,20)(45,20)
\qbezier(45,20)(50,20)(50,15)
\qbezier(50,15)(50,10)(45,10)
\qbezier(45,10)(40,10)(36,14)
\qbezier(33,17)(30,25)(30,30)

\ifnum#2=2
\put(3,28){\path(0,0)(2,2)(4,0)}
\put(28,28){\path(0,0)(2,2)(4,0)}
\put(38,15){\makebox{${\Huge c_{1}}$}}
\fi

\end{picture}
}

\def\rmoveiio#1#2{%#2=2 oriented #2=1 non-oriented #2=3 cohearent oriented
\setlength{\unitlength}{#1}
\begin{picture}(60,40)
\put(1,0){\line(0,1){40}}
\put(11,0){\line(0,1){40}}

{\allinethickness{.8pt}
\put(18,20){\vector(1,0){19}}
\put(37,20){\vector(-1,0){19}}}

\qbezier(40,0)(80,20)(40,40)

\qbezier(61,0)(57,3)(53,6)
\qbezier(50,8)(32,20)(50,32)
\qbezier(53,35)(56,37)(61,40)

\ifnum#2=2
\put(2,37){\path(0,0)(3,3)(6,0)}
\put(12,37){\path(0,0)(3,3)(6,0)}
\put(40,37){\path(0,0)(0,3)(3,3)}
\put(60,37){\path(0,0)(0,3)(-3,3)}
\put(47,25){\makebox{${\Huge c_{1}}$}}
\put(47,13){\makebox{${\Huge c_{2}}$}}
\fi

\ifnum#2=3
\put(2,2){\path(0,0)(3,-3)(6,0)}
\put(12,37){\path(0,0)(3,3)(6,0)}
\put(40,3){\path(0,0)(0,-3)(3,-3)}
\put(60,37){\path(0,0)(0,3)(-3,3)}
\put(47,25){\makebox{${\Huge c_{1}}$}}
\put(47,13){\makebox{${\Huge c_{2}}$}}
\fi

\end{picture}
}

\def\rmoveiiio#1#2{%#2=2 oriented #2=1 non-oriented #2=3 cohearent oriented
\setlength{\unitlength}{#1}
\begin{picture}(75,30)
\put(0,0){\line(1,1){15}}
\qbezier(15,15)(20,20)(20,30)

\put(10,0){\line(-1,1){4}}
\qbezier(4,6)(-5,15)(5,25)
\put(5,25){\line(1,1){5}}

\qbezier(20,0)(20,10)(16,14)
\put(14,16){\line(-1,1){8}}
\put(4,26){\line(-1,1){4}}

{\allinethickness{.8pt}
\put(28,15){\vector(1,0){15}}
\put(43,15){\vector(-1,0){15}}}

\qbezier(50,0)(50,10)(55,15)
\put(55,15){\line(1,1){15}}

\put(60,0){\line(1,1){5}}
\qbezier(65,5)(75,15)(66,24)
\put(64,26){\line(-1,1){4}}

\put(70,0){\line(-1,1){4}}
\put(64,6){\line(-1,1){8}}
\qbezier(54,16)(50,20)(50,30)

\ifnum#2=2
\put(0,27){\path(0,0)(0,3)(3,3)}
\put(7,30){\path(0,0)(3,0)(3,-3)}
\put(17,27){\path(0,0)(3,3)(6,0)}

\put(10,23){\makebox{${\Huge c_{1}}$}}
\put(10,3){\makebox{${\Huge c_{2}}$}}
\put(20,13){\makebox{${\Huge c_{3}}$}}

\put(47,27){\path(0,0)(3,3)(6,0)}
\put(60,27){\path(0,0)(0,3)(3,3)}
\put(67,30){\path(0,0)(3,0)(3,-3)}

\put(70,23){\makebox{${\Huge c'_{2}}$}}
\put(70,3){\makebox{${\Huge c'_{1}}$}}
\put(60,13){\makebox{${\Huge c'_{3}}$}}
\fi

\end{picture}
}

\def\rmovevio#1#2{%#2=2 oriented #2=1 non-oriented
\setlength{\unitlength}{#1}
\begin{picture}(50,30)
\put(5,0){\line(0,1){30}}

{\allinethickness{.8pt}
\put(10,15){\vector(1,0){15}}
\put(25,15){\vector(-1,0){15}}}

\qbezier(30,0)(30,20)(45,20)
\qbezier(45,20)(50,20)(50,15)
\qbezier(50,15)(50,10)(45,10)
\qbezier(45,10)(30,10)(30,30)

\put(34,15){\circle{5}}

\ifnum#2=2
\put(3,28){\path(0,0)(2,2)(4,0)}
\put(28,28){\path(0,0)(2,2)(4,0)}
\put(38,15){\makebox{${\Huge c_{1}}$}}
\fi

\end{picture}
}

\def\rmoveviio#1#2{%#2=2 oriented #2=1 non-oriented #2=3 cohearent oriented
\setlength{\unitlength}{#1}
\begin{picture}(60,40)
\put(1,0){\line(0,1){40}}
\put(11,0){\line(0,1){40}}

{\allinethickness{.8pt}
\put(18,20){\vector(1,0){19}}
\put(37,20){\vector(-1,0){19}}}

\qbezier(40,0)(80,20)(40,40)
\qbezier(60,0)(20,20)(60,40)
\put(50,6){\circle{5}}
\put(50,34){\circle{5}}

\ifnum#2=2
\put(2,37){\path(0,0)(3,3)(6,0)}
\put(12,37){\path(0,0)(3,3)(6,0)}
\put(40,37){\path(0,0)(0,3)(3,3)}
\put(60,37){\path(0,0)(0,3)(-3,3)}
\put(47,25){\makebox{${\Huge c_{1}}$}}
\put(47,13){\makebox{${\Huge c_{2}}$}}
\fi

\ifnum#2=3
\put(2,2){\path(0,0)(3,-3)(6,0)}
\put(12,37){\path(0,0)(3,3)(6,0)}
\put(40,3){\path(0,0)(0,-3)(3,-3)}
\put(60,37){\path(0,0)(0,3)(-3,3)}
\put(47,25){\makebox{${\Huge c_{1}}$}}
\put(47,13){\makebox{${\Huge c_{2}}$}}
\fi

\end{picture}
}

\def\rmoveviiio#1#2{%#2=2 oriented #2=1 non-oriented #2=3 cohearent oriented
\setlength{\unitlength}{#1}
\begin{picture}(70,30)
\put(0,0){\line(1,1){15}}
\qbezier(15,15)(20,20)(20,30)

\put(10,0){\line(-1,1){5}}
\qbezier(5,5)(-5,15)(5,25)
\put(5,25){\line(1,1){5}}

\qbezier(20,0)(20,10)(15,15)
\put(15,15){\line(-1,1){15}}
\put(5,5){\circle{5}}
\put(15,15){\circle{5}}
\put(5,25){\circle{5}}

{\allinethickness{.8pt}
\put(28,15){\vector(1,0){15}}
\put(43,15){\vector(-1,0){15}}}

\qbezier(50,0)(50,10)(55,15)
\put(55,15){\line(1,1){15}}

\put(60,0){\line(1,1){5}}
\qbezier(65,5)(75,15)(65,25)
\put(65,25){\line(-1,1){5}}

\put(70,0){\line(-1,1){15}}
\qbezier(55,15)(50,20)(50,30)

\put(65,5){\circle{5}}
\put(55,15){\circle{5}}
\put(65,25){\circle{5}}

\ifnum#2=2
\put(0,27){\path(0,0)(0,3)(3,3)}
\put(7,30){\path(0,0)(3,0)(3,-3)}
\put(17,27){\path(0,0)(3,3)(6,0)}

\put(20,13){\makebox{${\Huge c_{1}}$}}

\put(47,27){\path(0,0)(3,3)(6,0)}
\put(60,27){\path(0,0)(0,3)(3,3)}
\put(67,30){\path(0,0)(3,0)(3,-3)}

\put(60,13){\makebox{${\Huge c'_{1}}$}}
\fi

\end{picture}
}

\def\rmovevivo#1#2{%#2=2 oriented #2=1 non-oriented #2=3 cohearent oriented
\setlength{\unitlength}{#1}
\begin{picture}(70,30)
\put(0,0){\line(1,1){15}}
\qbezier(15,15)(20,20)(20,30)

\put(10,0){\line(-1,1){5}}
\qbezier(5,5)(-5,15)(5,25)
\put(5,25){\line(1,1){5}}

\qbezier(20,0)(20,10)(15,15)
\put(15,15){\line(-1,1){9}}
\put(4,26){\line(-1,1){4}}

\put(5,5){\circle{5}}
\put(15,15){\circle{5}}
{\allinethickness{.8pt}
\put(28,15){\vector(1,0){15}}
\put(43,15){\vector(-1,0){15}}}

\qbezier(50,0)(50,10)(55,15)
\put(55,15){\line(1,1){15}}

\put(60,0){\line(1,1){5}}
\qbezier(65,5)(75,15)(65,25)
\put(65,25){\line(-1,1){5}}

\put(70,0){\line(-1,1){4}}
\put(64,6){\line(-1,1){9}}
\qbezier(55,15)(50,20)(50,30)

\put(55,15){\circle{5}}
\put(65,25){\circle{5}}

\ifnum#2=2
\put(0,27){\path(0,0)(0,3)(3,3)}
\put(7,30){\path(0,0)(3,0)(3,-3)}
\put(17,27){\path(0,0)(3,3)(6,0)}

\put(20,13){\makebox{${\Huge c_{1}}$}}

\put(47,27){\path(0,0)(3,3)(6,0)}
\put(60,27){\path(0,0)(0,3)(3,3)}
\put(67,30){\path(0,0)(3,0)(3,-3)}

\put(60,13){\makebox{${\Huge c'_{1}}$}}
\fi

\end{picture}
}

\def\rmovetti#1{
\setlength{\unitlength}{#1}
\begin{picture}(60,20)
\put(0,0){\line(1,1){20}}
\put(20,0){\line(-1,1){20}}
\put(10,10){\circle{5}}
\put(17,13){\line(-1,1){4}}

{\allinethickness{.8pt}
\put(23,10){\vector(1,0){14}}
\put(37,10){\vector(-1,0){14}}}

\put(40,0){\line(1,1){20}}
\put(60,0){\line(-1,1){20}}
\put(50,10){\circle{5}}
\put(47,3){\line(-1,1){4}}
\end{picture}
}

\def\rmovetiio#1{
\setlength{\unitlength}{#1}
\begin{picture}(40,20)
\put(5,0){\line(0,1){20}}
{%\linethickness{2pt}
\put(2.,7){\line(1,0){6}} %bar
\put(2.,13){\line(1,0){6}} }%bar
{\allinethickness{.8pt}
\put(13,10){\vector(1,0){14}}
\put(27,10){\vector(-1,0){14}}}

\put(35,0){\line(0,1){20}}
\end{picture}
}

\def\rmovetiiio#1#2{%#2=2 oriented #2=1 non-oriented #2=3 cohearent oriented
\setlength{\unitlength}{#1}
\begin{picture}(70,20)

\put(0,0){\line(1,1){20}}
\put(20,0){\line(-1,1){9}}
\put(9,11){\line(-1,1){9}}

\put(7,3){\line(-1,1){4}}
\put(13,3){\line(1,1){4}}
\put(3,13){\line(1,1){4}}
\put(17,13){\line(-1,1){4}}

{\allinethickness{.8pt}
\put(23,10){\vector(1,0){14}}
\put(37,10){\vector(-1,0){14}}}

\qbezier(40,5)(48,20)(54,11)
\qbezier(56, 9)(62,0)(70,15)

\qbezier(40,15)(48,0)(55,10)
\qbezier(55,10)(62,20)(70,5)

\put(43,10){\circle{5}}
\put(67,10){\circle{5}}

\ifnum#2=2
\put(0,17){\path(0,0)(0,3)(3,3)}
\put(17,20){\path(0,0)(3,0)(3,-3)}

\put(15,8){\makebox{${\Huge c_{1}}$}}

\put(40,12){\path(0,0)(0,3)(3,3)}
\put(67,15){\path(0,0)(3,0)(3,-3)}

\put(53,3){\makebox{${\Huge c'_{1}}$}}
\fi

\end{picture}
}

\def\wmove#1#2{%#2=2 oriented #2=1 non-oriented #2=3 cohearent oriented
\setlength{\unitlength}{#1}
\begin{picture}(75,30)
\put(0,0){\line(1,1){15}}
\qbezier(15,15)(20,20)(20,30)

\put(10,0){\line(-1,1){4}}
\qbezier(4,6)(-5,15)(5,25)
\put(5,25){\line(1,1){5}}

\put(5,25){\circle{5}}

\qbezier(20,0)(20,10)(16,14)
\put(14,16){\line(-1,1){14}}
{\allinethickness{.8pt}
\put(28,15){\vector(1,0){15}}
\put(43,15){\vector(-1,0){15}}}

\qbezier(50,0)(50,10)(55,15)
\put(55,15){\line(1,1){15}}

\put(60,0){\line(1,1){5}}
\qbezier(65,5)(75,15)(66,24)
\put(64,26){\line(-1,1){4}}

\put(70,0){\line(-1,1){14}}
\qbezier(54,16)(50,20)(50,30)

\put(65,5){\circle{5}}

\ifnum#2=2
\put(0,27){\path(0,0)(0,3)(3,3)}
\put(7,30){\path(0,0)(3,0)(3,-3)}
\put(17,27){\path(0,0)(3,3)(6,0)}

\put(10,23){\makebox{${\Huge c_{1}}$}}
\put(10,3){\makebox{${\Huge c_{2}}$}}
\put(20,13){\makebox{${\Huge c_{3}}$}}

\put(47,27){\path(0,0)(3,3)(6,0)}
\put(60,27){\path(0,0)(0,3)(3,3)}
\put(67,30){\path(0,0)(3,0)(3,-3)}

\put(70,23){\makebox{${\Huge c'_{2}}$}}
\put(70,3){\makebox{${\Huge c'_{1}}$}}
\put(60,13){\makebox{${\Huge c'_{3}}$}}
\fi

\end{picture}
}

\def\armoveio#1#2{%#2=2 oriented #2=1 non-oriented
\setlength{\unitlength}{#1}
\begin{picture}(70,30)

\put(2,0){\line(0,1){30}}
{\allinethickness{.8pt}
\put(5,0){\line(0,1){30}}}
\put(8,0){\line(0,1){30}}

{\allinethickness{.8pt}
\put(15,15){\vector(1,0){15}}
\put(30,15){\vector(-1,0){15}}}

\qbezier(37,0)(37,10)(40,15)
\qbezier(40,15)(37,20)(37,30)
{\allinethickness{.8pt}
\qbezier(40,0)(40,20)(55,20)
\qbezier(55,20)(60,20)(60,15)
\qbezier(60,15)(60,10)(55,10)
\qbezier(55,10)(50,10)(46,14)
\qbezier(43,17)(40,25)(40,30)
}
\qbezier(43,0)(43,7)(45,10)
\qbezier(45,10)(63,3)(63,15)
\qbezier(63,15)(63,26)(45,21)
\qbezier(45,21)(43,23)(43,30)

\qbezier(48,15)(57,11)(57,15)
\qbezier(57,15)(57,18)(48,15)

\ifnum#2=2
\put(3,28){\path(0,0)(2,2)(4,0)}
\put(38,28){\path(0,0)(2,2)(4,0)}
\put(47,15){\makebox{${\Huge c_{1}}$}}
\fi

\end{picture}
}

\def\armoveiio#1#2{%#2=2 oriented #2=1 non-oriented #2=3 cohearent oriented
\setlength{\unitlength}{#1}
\begin{picture}(80,40)
\put(2,0){\line(0,1){40}}
{\allinethickness{.8pt}
\put(5,0){\line(0,1){40}}}
\put(8,0){\line(0,1){40}}

\put(15,0){\line(0,1){40}}
{\allinethickness{.8pt}
\put(18,0){\line(0,1){40}}}
\put(21,0){\line(0,1){40}}

{\allinethickness{.8pt}
\put(25,20){\vector(1,0){15}}
\put(40,20){\vector(-1,0){15}}}

\qbezier(44,0)(51,3)(56,7)
\qbezier(56,7)(40,20)(56,33)
\qbezier(56,33)(52,36)(44,40)

{\allinethickness{.8pt}
\qbezier(50,0)(90,20)(50,40)

\qbezier(71,0)(67,3)(63,6)
\qbezier(60,8)(42,20)(60,32)
\qbezier(63,35)(66,37)(71,40)}

\qbezier(77,0)(72,3)(66,7)
\qbezier(66,7)(82,20)(66,33)
\qbezier(66,33)(72,37)(77,40)

\qbezier(56,0)(58,2)(61,3)
\qbezier(61,3)(64,2)(67,0)

\qbezier(56,40)(58,38)(61,37)
\qbezier(61,37)(64,38)(67,40)

\qbezier(60,12)(48,20)(60,28)
\qbezier(60,12)(72,20)(60,28)

\ifnum#2=2
\put(5,37){\path(0,0)(3,3)(6,0)}
\put(18,37){\path(0,0)(3,3)(6,0)}
\put(50,37){\path(0,0)(0,3)(3,3)}
\put(70,37){\path(0,0)(0,3)(-3,3)}
\put(57,25){\makebox{${\Huge c_{1}}$}}
\put(57,13){\makebox{${\Huge c_{2}}$}}
\fi

\ifnum#2=3
\put(2,2){\path(0,0)(3,-3)(6,0)}
\put(12,37){\path(0,0)(3,3)(6,0)}
\put(40,3){\path(0,0)(0,-3)(3,-3)}
\put(60,37){\path(0,0)(0,3)(-3,3)}
\put(47,25){\makebox{${\Huge c_{1}}$}}
\put(47,13){\makebox{${\Huge c_{2}}$}}
\fi

\end{picture}
}
\def\armoveiiio#1#2{%#2=2 oriented #2=1 non-oriented #2=3 cohearent oriented
\setlength{\unitlength}{#1}
\begin{picture}(90,30)

\put(2,0){\line(1,1){5}}
\qbezier(7,5)(-2,15)(7,25)
\put(7,25){\line(-1,1){5}}

\put(8,0){\line(1,1){2}}
\put(10,2){\line(1,-1){2}}

\put(8,30){\line(1,-1){2}}
\put(10,28){\line(1,1){2}}

\put(10,8){\line(1,1){7}}
\put(17,15){\line(-1,1){7}}
\qbezier(10,8)(3,15)(10,22)

\put(18,30){\line(-1,-1){5}}
\put(13,25){\line(1,-1){6}}
\qbezier(19,19)(23,21)(23,30)

\put(18,0){\line(-1,1){5}}
\put(13,5){\line(1,1){6}}
\qbezier(19,11)(23,13)(23,0)

\qbezier(27,0)(27,10)(23,15)
\qbezier(23,15)(27,20)(27,30)

{\allinethickness{.8pt}
\put(5,0){\line(1,1){15}}
\qbezier(20,15)(25,20)(25,30)

\put(15,0){\line(-1,1){4}}
\qbezier(9,6)(0,15)(10,25)
\put(10,25){\line(1,1){5}}

\qbezier(25,0)(25,10)(21,14)
\put(19,16){\line(-1,1){8}}
\put(9,26){\line(-1,1){4}}
}
{\allinethickness{.8pt}
\put(35,15){\vector(1,0){15}}
\put(50,15){\vector(-1,0){15}}}
\qbezier(58,0)(58,10)(62,15)
\qbezier(62,15)(58,20)(58,30)

\qbezier(62,0)(62,9)(65,12)
\put(65,12){\line(1,-1){7}}
\put(72,5){\line(-1,-1){5}}

\qbezier(62,30)(62,21)(65,18)
\put(65,18){\line(1,1){7}}
\put(72,25){\line(-1,1){5}}

\put(68,15){\line(1,-1){7}}
\qbezier(75,8)(81,15)(75,22)
\put(75,22){\line(-1,-1){7}}

\put(73,30){\line(1,-1){2}}
\put(75,28){\line(1,1){2}}

\put(73,0){\line(1,1){2}}
\put(75,2){\line(1,-1){2}}

\put(83,0){\line(-1,1){5}}
\qbezier(78,5)(87,15)(78,25)
\put(78,25){\line(1,1){5}}

{\allinethickness{.8pt}
\qbezier(60,0)(60,10)(65,15)
\put(65,15){\line(1,1){15}}

\put(70,0){\line(1,1){5}}
\qbezier(75,5)(85,15)(76,24)
\put(74,26){\line(-1,1){4}}

\put(80,0){\line(-1,1){4}}
\put(74,6){\line(-1,1){8}}
\qbezier(64,16)(60,20)(60,30)
}

\ifnum#2=2
\put(5,27){\path(0,0)(0,3)(3,3)}
\put(12,30){\path(0,0)(3,0)(3,-3)}
\put(22,27){\path(0,0)(3,3)(6,0)}

\put(15,23){\makebox{${\Huge c_{1}}$}}
\put(15,3){\makebox{${\Huge c_{2}}$}}
\put(25,13){\makebox{${\Huge c_{3}}$}}

\put(57,27){\path(0,0)(3,3)(6,0)}
\put(70,27){\path(0,0)(0,3)(3,3)}
\put(77,30){\path(0,0)(3,0)(3,-3)}

\put(80,23){\makebox{${\Huge c'_{2}}$}}
\put(80,3){\makebox{${\Huge c'_{1}}$}}
\put(70,13){\makebox{${\Huge c'_{3}}$}}
\fi

\end{picture}
}

\def\armovevio#1#2{%#2=2 oriented #2=1 non-oriented
\setlength{\unitlength}{#1}
\begin{picture}(60,30)
\put(2,0){\line(0,1){30}}
{\allinethickness{.8pt}
\put(5,0){\line(0,1){30}}}
\put(8,0){\line(0,1){30}}

{\allinethickness{.8pt}
\put(12,15){\vector(1,0){15}}
\put(27,15){\vector(-1,0){15}}}

\qbezier(32,0)(32,23)(53,23)
\qbezier(53,23)(58,23)(58,15)
\qbezier(58,15)(58,4)(40,9)
\qbezier(34,13)(32,18)(32,30)
{\allinethickness{.8pt}
\qbezier(35,0)(35,20)(50,20)
\qbezier(50,20)(55,20)(55,15)
\qbezier(55,15)(55,8)(41,12)
\qbezier(37,17)(35,20)(35,30)}

\qbezier(38,0)(38,17)(50,17)
\qbezier(50,17)(52,17)(52,15)
\qbezier(52,15)(52,12)(43,14)
\qbezier(39,19)(38,22)(38,30)

\ifnum#2=2
\put(3,28){\path(0,0)(2,2)(4,0)}
\put(33,28){\path(0,0)(2,2)(4,0)}
\put(43,15){\makebox{${\Huge c_{1}}$}}
\fi

\end{picture}
}

\def\armoveviio#1#2{%#2=2 oriented #2=1 non-oriented #2=3 cohearent oriented
\setlength{\unitlength}{#1}
\begin{picture}(80,40)
\put(2,0){\line(0,1){40}}
{\allinethickness{.8pt}
\put(5,0){\line(0,1){40}}}
\put(8,0){\line(0,1){40}}

\put(15,0){\line(0,1){40}}
{\allinethickness{.8pt}
\put(18,0){\line(0,1){40}}}
\put(21,0){\line(0,1){40}}

{\allinethickness{.8pt}
\put(25,20){\vector(1,0){15}}
\put(40,20){\vector(-1,0){15}}}

\qbezier(44,0)(88,20)(44,40)
\qbezier(56,0)(92,20)(56,40)

{\allinethickness{.8pt}
\qbezier(50,0)(90,20)(50,40)
\qbezier(71,0)(67,3)(64,5)
\qbezier(59,9)(42,20)(59,31)
\qbezier(64,35)(67,37)(71,40)
 }

\put(65,0){\line(-3,2){4}}
\put(77,0){\line(-3,2){10.5}}

\qbezier(56,7)(38,20)(56,33)
\qbezier(61,11)(46,20)(61,29)

\put(61,37){\line(3,2){5}}
\put(67,33){\line(3,2){10}}

\ifnum#2=2
\put(2,37){\path(0,0)(3,3)(6,0)}
\put(15,37){\path(0,0)(3,3)(6,0)}
\put(50,37){\path(0,0)(0,3)(3,3)}
\put(70,37){\path(0,0)(0,3)(-3,3)}
\put(57,25){\makebox{${\Huge c_{1}}$}}
\put(57,13){\makebox{${\Huge c_{2}}$}}
\fi

\ifnum#2=3
\put(2,2){\path(0,0)(3,-3)(6,0)}
\put(12,37){\path(0,0)(3,3)(6,0)}
\put(40,3){\path(0,0)(0,-3)(3,-3)}
\put(60,37){\path(0,0)(0,3)(-3,3)}
\put(47,25){\makebox{${\Huge c_{1}}$}}
\put(47,13){\makebox{${\Huge c_{2}}$}}
\fi

\end{picture}
}

\def\armoveviiio#1#2{%#2=2 oriented #2=1 non-oriented #2=3 cohearent oriented
\setlength{\unitlength}{#1}
\begin{picture}(90,30)
\put(2,0){\line(1,1){15}}
\qbezier(17,15)(23,21)(23,30)
\put(8,0){\line(1,1){15}}
\qbezier(23,15)(28,21)(28,30)

\put(18,0){\line(-1,1){5}}
\put(12,0){\line(-1,1){2}}

\qbezier(10,8)(3,15)(10,22)
\put(10,22){\line(1,1){8}}
\qbezier(7,5)(-2,15)(7,25)
\put(7,25){\line(1,1){5}}

\qbezier(23,0)(23,8)(20,12)
\qbezier(27,0)(27,10)(23,15)
\put(20,18){\line(-1,1){7}}
\put(17,15){\line(-1,1){7}}
\put(10,28){\line(-1,1){2}}
\put(7,25){\line(-1,1){5}}

{\allinethickness{.8pt}
\put(5,0){\line(1,1){15}}
\qbezier(20,15)(25,20)(25,30)

\put(15,0){\line(-1,1){3.5}}
\qbezier(8.5,7)(0,15)(10,25)
\put(10,25){\line(1,1){5}}

\qbezier(25,0)(25,10)(21.5,13.)
\put(18.5,16.5){\line(-1,1){7}}
\put(8.5,26.5){\line(-1,1){4}}
}
{\allinethickness{.8pt}
\put(35,15){\vector(1,0){15}}
\put(50,15){\vector(-1,0){15}}}
\qbezier(58,0)(58,10)(62,15)
\put(62,15){\line(1,1){15}}

\qbezier(62,0)(62,9)(65,12)
\put(65,12){\line(1,1){18}}

\put(67,0){\line(1,1){8}}
\put(73,0){\line(1,1){5}}
\qbezier(75,8)(81,15)(75,22)
\qbezier(78,5)(87,15)(78,25)
\put(72,25){\line(-1,1){5}}
\put(75,28){\line(-1,1){2}}

\put(77,0){\line(-1,1){2}}
\put(83,0){\line(-1,1){5}}
\put(72,5){\line(-1,1){7}}
\put(75,8){\line(-1,1){7}}
\qbezier(62,15)(58,20)(58,30)
\qbezier(62,30)(62,21)(65,18)

{\allinethickness{.8pt}
\qbezier(60,0)(60,10)(65,15)
\put(65,15){\line(1,1){15}}

\put(70,0){\line(1,1){5}}
\qbezier(75,5)(85,15)(76.7,23.5)
\put(73.5,26.5){\line(-1,1){3.5}}

\put(80,0){\line(-1,1){3.5}}
\put(73.5,6.5){\line(-1,1){7}}
\qbezier(63,16.5)(60,20)(60,30)
}

\ifnum#2=2
\put(5,27){\path(0,0)(0,3)(3,3)}
\put(12,30){\path(0,0)(3,0)(3,-3)}
\put(22,27){\path(0,0)(3,3)(6,0)}

\put(57,27){\path(0,0)(3,3)(6,0)}
\put(70,27){\path(0,0)(0,3)(3,3)}
\put(77,30){\path(0,0)(3,0)(3,-3)}

\fi

\end{picture}
}

\def\armovevivo#1#2{%#2=2 oriented #2=1 non-oriented #2=3 cohearent oriented
\setlength{\unitlength}{#1}
\begin{picture}(90,30)
\put(2,0){\line(1,1){15}}
\qbezier(17,15)(23,21)(23,30)
\put(8,0){\line(1,1){15}}
\qbezier(23,15)(28,21)(28,30)

\put(18,0){\line(-1,1){5}}
\put(12,0){\line(-1,1){2}}
\qbezier(10,8)(3,15)(10,22)
\put(13,25){\line(1,1){5}}
\qbezier(7,5)(-2,15)(7,25)
\put(10,28){\line(1,1){2}}

\qbezier(23,0)(23,8)(20,12)
\qbezier(27,0)(27,10)(23,15)
\put(20,18){\line(-1,1){7}}
\put(17,15){\line(-1,1){7}}
\put(10,28){\line(-1,1){2}}
\put(7,25){\line(-1,1){5}}

{\allinethickness{.8pt}
\put(5,0){\line(1,1){15}}
\qbezier(20,15)(25,20)(25,30)

\put(15,0){\line(-1,1){3.5}}
\qbezier(8.5,7)(0,15)(10,25)
\put(10,25){\line(1,1){5}}

\qbezier(25,0)(25,10)(21.5,13.)
\put(18.5,16.5){\line(-1,1){7}}
\put(8.5,26.5){\line(-1,1){4}}
}
{\allinethickness{.8pt}
\put(35,15){\vector(1,0){15}}
\put(50,15){\vector(-1,0){15}}}
\qbezier(58,0)(58,10)(62,15)
\put(62,15){\line(1,1){15}}

\qbezier(62,0)(62,9)(65,12)
\put(65,12){\line(1,1){18}}

\put(67,0){\line(1,1){5}}
\put(73,0){\line(1,1){2}}
\qbezier(75,8)(81,15)(75,22)
\qbezier(78,5)(87,15)(78,25)
\put(72,25){\line(-1,1){5}}
\put(75,28){\line(-1,1){2}}

\put(77,0){\line(-1,1){2}}
\put(83,0){\line(-1,1){5}}
\put(72,5){\line(-1,1){7}}
\put(75,8){\line(-1,1){7}}
\qbezier(62,15)(58,20)(58,30)
\qbezier(62,30)(62,21)(65,18)

{\allinethickness{.8pt}
\qbezier(60,0)(60,10)(65,15)
\put(65,15){\line(1,1){15}}

\put(70,0){\line(1,1){5}}
\qbezier(75,5)(85,15)(76.7,23.5)
\put(73.5,26.5){\line(-1,1){3.5}}

\put(80,0){\line(-1,1){3.5}}
\put(73.5,6.5){\line(-1,1){7}}
\qbezier(63,16.5)(60,20)(60,30)
}

\ifnum#2=2
\put(5,27){\path(0,0)(0,3)(3,3)}
\put(12,30){\path(0,0)(3,0)(3,-3)}
\put(22,27){\path(0,0)(3,3)(6,0)}

\put(57,27){\path(0,0)(3,3)(6,0)}
\put(70,27){\path(0,0)(0,3)(3,3)}
\put(77,30){\path(0,0)(3,0)(3,-3)}

\fi

\end{picture}
}

\def\armovetiiio#1#2{%#2=2 oriented #2=1 non-oriented #2=3 cohearent oriented
\setlength{\unitlength}{#1}
\begin{picture}(90,20)
{\allinethickness{.8pt}
\put(5,0){\line(1,1){20}}
\put(25,0){\line(-1,1){9}}
\put(14,11){\line(-1,1){9}}}

\put(1,0){\line(1,1){4}}
\qbezier(5,4)(7,6)(10,5)
\qbezier(10,5)(13,4)(15,6)
\qbezier(9,0)(11,2)(10,4)
\qbezier(9,6)(9,8)(11,10)

\qbezier(15,14)(17,16)(20,15)
\qbezier(20,15)(23,14)(25,16)
\put(25,16){\line(1,1){3}}
\qbezier(19,10)(21,12)(20,14)
\qbezier(19,16)(19,18)(21,20)

\put(29,0){\line(-1,1){4}}
\qbezier(25,4)(23,6)(20,5)
\qbezier(20,5)(17,4)(15,6)
\qbezier(21,0)(19,2)(20,4)
\qbezier(21,6)(21,8)(19,10)

\qbezier(15,14)(13,16)(10,15)
\qbezier(10,15)(7,14)(5,16)
\put(5,16){\line(-1,1){3}}
\qbezier(11,10)(9,12)(10,14)
\qbezier(11,16)(11,18)(9,20)

{\allinethickness{.8pt}
\put(30,10){\vector(1,0){15}}
\put(45,10){\vector(-1,0){15}}}

{\allinethickness{.8pt}
\qbezier(53,2)(63,22)(69,11)
\qbezier(71, 9)(75,3)(80,8)
\put(83.5,12.5){\line(2,3){3}}

\qbezier(59.5,8.5)(65,2)(70,10)
\put(56.5,12.5){\line(-2,3){3}}
\qbezier(70,10)(77,22)(87,2)

}

\qbezier(50,2)(62,25)(70,14)
\qbezier(70,14)(78,25)(90,2)

\qbezier(56,2)(63,18)(67,10)
\qbezier(67,10)(64,7)(61,10)
\put(58,14){\line(-2,3){2}}

\put(51,16.5){\line(2,-3){4.2}}
\qbezier(58.5,6.5)(64,0)(70,6)
\qbezier(70,6)(76,0)(82.,6.)
\put(85,10.){\line(2,3){4.2}}

\qbezier(84,2)(76,17)(73,10)
\qbezier(73,10)(76,7)(79,10)
\put(82,14){\line(2,3){2}}

\ifnum#2=2
\put(5,17){\path(0,0)(0,3)(3,3)}
\put(22,20){\path(0,0)(3,0)(3,-3)}

\put(20,8){\makebox{${\Huge c_{1}}$}}

\put(55,12){\path(0,0)(0,3)(3,3)}
\put(82,15){\path(0,0)(3,0)(3,-3)}

\put(68,3){\makebox{${\Huge c'_{1}}$}}
\fi

\end{picture}
}

\newcommand{\R}{\mathbb R}
\newcommand{\Z}{\mathbb Z}

\begin{document}

\title[Virtual links which are equivalent as twisted links]{Virtual links which are equivalent \\  as twisted links}

\author[Naoko Kamada]{Naoko Kamada}
\address{Graduate School of Natural Sciences,  Nagoya City University,
Mizuho-ku, Nagoya, Aichi 467-8501, Japan}
\email{kamada@nsc.nagoya-cu.ac.jp}

%\author[Seiichi Kamada]{Seiichi Kamada}
%\address{Department of Mathematics,
%Osaka City University,
%Sugimoto, Sumiyoshi-ku,
%Osaka 558-8585, Japan}
%\email{skamada@sci.osaka-cu.ac.jp}

\author[Seiichi Kamada]{Seiichi Kamada}
\address{Department of Mathematics,
Osaka University,
Toyonaka, Osaka 560-0043, Japan}
\email{kamada@math.sci.osaka-u.ac.jp}



\date{\today}

\begin{abstract}
A virtual link is a generalization of a classical link that is defined as an equivalence class of certain diagrams, called virtual link diagrams. It is further generalized to a twisted link. Twisted links are in one-to-one correspondence with stable equivalence classes of links in oriented thickenings of (possibly non-orientable) closed surfaces. By definition, equivalent virtual links are also equivalent as twisted links. In this paper, we discuss when two virtual links are equivalent as twisted links, and give a necessary and sufficient condition for this to be the case.
\end{abstract}

\maketitle

\section{Introduction}
\label{sect:intro}

A virtual link is a generalization of a classical link introduced by Kauffman \cite{rkauD}, which is defined as an equivalence class of certain diagrams, called virtual link diagrams.
Virtual link theory is quite natural when we discuss Gauss chord diagrams, since every Gauss chord diagram is realized as a virtual link diagram up to virtual Reidemeister moves \cite{rGPV, rkauD}. 
Moreover, virtual links are in one-to-one correspondence with  abstract links on oriented surfaces \cite{rkk1},
and in one-to-one correspondence
with stable equivalence classes of links in oriented thickenings of oriented closed surfaces \cite{rCKS, rkk1}.
It is known that the set of classical links is a subset of the set of virtual links, i.e., two classical link diagrams are equivalent  as virtual links  if and only if they are equivalent as classical links \cite{rGPV, rkauD, rKup}.  This fact is obtained by considering knot groups with peripheral structures \cite{rGPV}, or by assuming a stronger fact due to Kuperburg \cite{rKup} that a virtual link has a unique irreducible representative as a link in an oriented thickening of an oriented surface.
For details and related topics on virtual knot theory, refer to \cite{rDye, rGPV, rkk1, rkauD, rMant}.

Bourgoin \cite{rbor} generalized virtual links to twisted links. Twisted links are in one-to-one correspondence with abstract links  on  surfaces \cite{rbor, rkk7},
and in one-to-one correspondence
with  stable  equivalence classes of links in oriented thickenings of closed surfaces \cite{rbor, rkk7}.

A {\it virtual link diagram} is a link diagram in $\R^2$  that   may have some {\it virtual crossings}, which are crossings without over/under information but which are decorated with a small circle surrounding it.
A  {\it twisted link diagram\/}
is a virtual link diagram possibly with {\it bars\/} on arcs.
Referring to Figure \ref{fgrmoves}, the  moves R1, R2, R3 are called {\it classical Reidemeister moves},
the moves V1, \dots, V4 are called {\it virtual Reidemeister moves}, and the moves T1, T2, T3 are called {\it twisted Reidemeister moves}. All of these are called {\it extended Reidemeister moves}.

A {\it virtual link} is an equivalence class of virtual link diagrams under classical and virtual Reidemeister moves.
A {\it twisted link} is an equivalence class of twisted link diagrams under extended Reidemeister moves.

\begin{figure}[h]
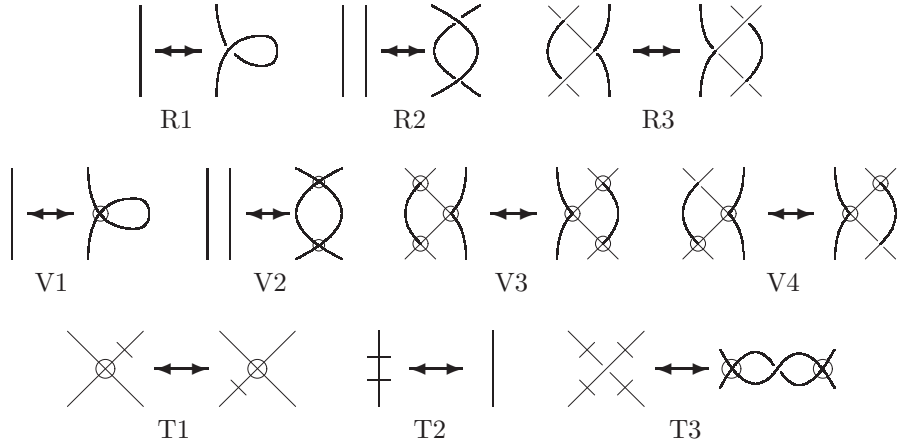

\centerline{
\rmoveio{.4mm}{1}\hspace{.6cm}\rmoveiio{.30mm}{1}\hspace{.7cm}\rmoveiiio{.4mm}{1}}
\centerline{
R1 \hspace{2.4cm} R2 \hspace{2.6cm} R3 \phantom{MM}
}

\vspace{0.5cm}

\centerline{
\rmovevio{.4mm}{1}\hspace{.5cm}\rmoveviio{.30mm}{1}\hspace{.6cm}\rmoveviiio{.4mm}{1}
\hspace{.5cm}\rmovevivo{.4mm}{1}
}
\centerline{
V1 \hspace{2.2cm} V2 \hspace{2.5cm} V3 \hspace{2.9cm} V4 \phantom{MM}
}

\vspace{0.5cm}

\centerline{
\rmovetti{.5mm}\hspace{.5cm}\rmovetiio{.5mm}\hspace{.5cm}\rmovetiiio{.5mm}{1}
}

\centerline{
T1 \hspace{2.7cm} T2 \hspace{2.7cm} T3 \phantom{M}
}
\caption{Classical, virtual and twisted Reidemeister moves}\label{fgrmoves}
\end{figure}

A geometric interpretation for a twisted link diagram $D$ is obtained by considering its associated abstract link diagram $A(D)$
as shown in Figure~\ref{fig:figabst}, which is a link diagram  on a compact surface \cite{rbor, rkk1, rkk7}.  The figure shows the local correspondence between $D$ and $A(D)$, and an example. 
Note that a bar of $D$ implies a half-twist of the ambient surface of $A(D)$.

\begin{figure}[h]
\centerline{
\includegraphics[width=9cm]{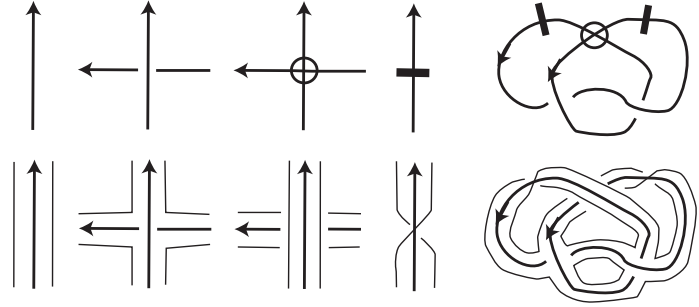}
}
\vspace{-0.2cm}
\caption{A twisted link diagram and its associated abstract link diagram}
\label{fig:figabst} 
\end{figure}

%\vspace{0.5cm} 
%\begin{figure}[ht]
%\begin{center}
%\includegraphics[width=7cm]{figdiagd.eps}
%\caption{A twisted link diagram and its associated abstract link diagram}
%\label{fig:figdiagd} 
%\end{center}
%\end{figure}

By definition, virtual link diagrams are twisted link diagrams, and if two virtual link diagrams are equivalent as virtual links then they are equivalent as twisted links. Thus the inclusion map
$$\iota : \{\text{virtual link diagrams}\} \rightarrow \{\text{twisted link diagrams}\}$$
yields a natural map
\begin{equation*}
f : \{\text{virtual links}\} \rightarrow \{\text{twisted links}\}.
\end{equation*}

In this paper, we discuss when two elements are mapped to the same element by $f$,
and give a necessary and sufficient condition for this to be the case.
This clarifies a remark made in \cite[p.1251]{rbor}, which claims  that virtual link theory injects into the theory of links in oriented thickenings; see Remark \ref{rmk-rbor}. 

For a virtual link $L$,  let $s(L)$ denote the virtual link represented by a
diagram $s(D)$  that   is obtained from a diagram $D$ of $L$ by a reflection along a line  in  $\mathbb{R}^2$ and by switching over/under information on all classical crossings.  See Figure~\ref{fig:fgr}, where $r$ is a reflection along a line  in  $\mathbb{R}^2$ and $c$ is  switching over/under information.

\begin{figure}[h]
\centerline{
\includegraphics[width=8.0cm]{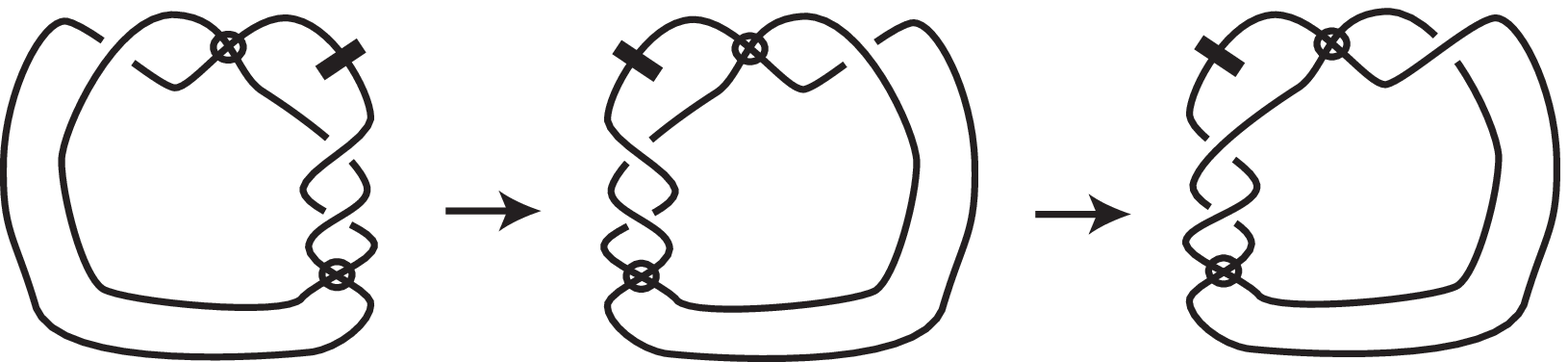}
}
\setlength{\unitlength}{1mm}
\begin{picture}(90,1)(0,0)
\put(14,0){$D$}
\put(40,0){$r(D)$}
\put(66,0){$s(D) = c \circ r (D)$}
\put(29,15){$r$}
\put(59,15){$c$}
\end{picture}
\caption{}\label{fig:fgr}
\end{figure}

\begin{thm}\label{thm1}
Two virtual knots $L$ and  $L'$ are equivalent as twisted knots if and only if
$L'$ is equivalent to $L$ or $s(L)$ as a virtual knot.
\end{thm}

This theorem is a special case of our main theorem (Theorem~\ref{thm3}),  stated in Section~\ref{sect:def}.

It is known that there is a virtual knot $L$ such that $L$ and $s(L)$ are not equivalent as virtual knots \cite{rSaw}.
Thus the map $f$ is not injective.

A link diagram (without virtual crossings nor bars) is referred to as a
{\it classical link diagram}, and a {\it classical link} means an equivalence class of classical link diagrams under classical Reidemeister moves. 
Recall that the set of classical links is a subset of the set of virtual links. It is also a subset of the set of twisted links. 
As we will see below the following holds. 

\begin{thm}\label{thm2}
The map $f$ restricted to the set of classical links is injective, i.e., two classical links are equivalent  as twisted links  if and only if they are equivalent as classical links.
\end{thm}

In this paper all (classical, virtual or twisted) links are oriented.  A link is called a knot if it consists of one component.
Although virtual links are equivalence classes of virtual link diagrams, we often say that two virtual links $L$ and $L'$ are {\it equivalent as virtual links} (or {\it as twisted links}, respectively) if their representatives are equivalent as virtual link diagrams (or as twisted link diagrams, respectively).  

The paper is organized as follows:
In Section~\ref{sect:def} we give necessary definitions and state the main results (Proposition~\ref{prop1} and Theorem~\ref{thm3}).  Proofs of the latter are given in Section~\ref{sect:proof}.
%In Section~\ref{sect:Gauss} we introduce Gauss chord diagrams for twisted links and apply them to give
%an alternative proof of Proposition~\ref{prop1}.

This work was supported by JSPS KAKENHI Grant Numbers JP15K04879, JP26287013,  JP19K03496 and JP19H01788.

\section{Definitions and the main theorem}
\label{sect:def}

Let $D$ be a (classical, virtual or twisted) diagram.
A {\it split decomposition} of $D$ is a collection of mutually disjoint subdiagrams $D_1, \dots, D_n$ such that
$D = D_1 \cup \dots \cup D_n$ for some $n \geq 1$.
We denote it by $D = D_1 \sqcup \dots \sqcup D_n$.

Let $L$ be a (classical, virtual or twisted) link.
A {\it split decomposition} of $L$ is a collection of sublinks $L_1, \dots, L_n$ such that
there is a diagram $D$ of $L$ with a
split decomposition
$D = D_1 \sqcup \dots \sqcup D_n$ such that $L_i$ is represented by $D_i$ for $i=1,\dots, n$.
We denote it by $L = L_1 \sqcup \dots \sqcup L_n$.
A (classical, virtual or twisted) link $L$ is {\it splittable} if
there is a split decomposition $L= L_1 \sqcup \dots \sqcup L_n$ with $n \geq 2$; otherwise, $L$ is {\it non-splittable}.
A split decomposition $L = L_1 \sqcup \dots \sqcup L_n$ is called {\it  maximal } if
for each $i=1,\dots, n$, $L_i$ is non-splittable. Note that a    maximal  split decomposition is unique up to reordering (Lemma~\ref{uniquemiximamdecomp}).

For a (classical, virtual or twisted) link diagram $D$ in $\mathbb{R}^2$,  as in Section~\ref{sect:intro} we let 
 $s(D)$ denote a diagram obtained from $D$ by a reflection along a line  in  $\mathbb{R}^2$ and switching over/under information on all classical crossings.   If $D$ and $D'$ are equivalent as (classical, virtual or twisted) link diagrams, so are $s(D)$ and $s(D')$.  Thus, for a (classical, virtual or twisted) link $L$, we have that $s(L)$ is well defined as a (classical, virtual or twisted) link.  Note  that while a classical link $L$ and its counterpart $s(L)$ are equivalent as classical links, a virtual link $L$ and  its counterpart $s(L)$ may not be equivalent as virtual links.

We prove the following proposition in Section~\ref{sect:proof}.

\begin{prop}\label{prop1}
For any twisted link $L$, we have that $L$ and $s(L)$ are equivalent as twisted links. \end{prop}

\begin{cor}\label{cor:vt}
For any virtual link $L$, we have that $L$ and $s(L)$ are equivalent  as twisted links. Thus, $f(L) = f(s(L))$.
\end{cor}

\begin{defn}
Two virtual links $L$ and $L'$ are {\it $s$-congruent} if there are maximal split decompositions
$L = L_1 \sqcup \dots \sqcup L_n$ and $L' = L'_1 \sqcup \dots \sqcup L'_n$ such that for each $i=1, \dots, n$,
$L'_i$ is equivalent to $L_i$ or $s(L_i)$ as a virtual link.
\end{defn}

The following is our main theorem.

\begin{thm}\label{thm3}
Let $L$ and $L'$ be virtual links.  Then $L$ and $L'$ are equivalent as twisted links if and only if they are $s$-congruent.
\end{thm}

Theorem~\ref{thm1} is a special case of Theorem~\ref{thm3}.
Theorem~\ref{thm2} follows from  Theorem~\ref{thm3}, since classical links $L$ and $L'$ are $s$-congruent if and only if they are equivalent as classical links.

\begin{rmk}\label{rmk-rbor}
In \cite[p.1251]{rbor} it is stated that virtual link theory injects into the theory of links in oriented thickenings. It should be understood that virtual link theory {\it modulo $s$-congruence} injects into twisted link theory.
There is an alternative proof of Theorem~\ref{thm3} using a uniqueness theorem (\cite[Theorem~1]{rbor}) of irreducible representatives of links in oriented thickenings of closed surfaces.
Our proof given in Section~\ref{sect:proof} is a direct argument using diagrams.
\end{rmk}

\section{Proofs}
\label{sect:proof}

\begin{lem}\label{uniquemiximamdecomp}
A  maximal  split decomposition is unique up to reordering. That is, if $L = L_1 \sqcup \dots \sqcup L_n$
and $L' = L'_1 \sqcup \dots \sqcup L'_{n'}$ are  maximal  split decompositions of equivalent (classical, virtual or twisted) links $L$ and $L'$,  then $n=n'$ and there exists a permutation $\sigma$ of $\{1, \dots, n\}$ such that for each  $i=1,\dots, n$, $L_i$ is equivalent to $L'_{\sigma(i)}$ as a (classical, virtual or twisted) link.
\end{lem}

\begin{proof}
Fix an equivalence between $L$ and $L'$.
(An {\it equivalence} between $L$ and $L'$ is a sequence of diagrams $D=D^0, D^1, D^2, \dots, D^m =D'$ for some $m$ such that $D$ and $D'$ are diagrams of $L$ and $L'$,  respectively,  and where $D^{k+1}$, $k=0, 1, \dots, m-1$, is obtained from $D^k$ by a single extended Reidemeister move.  Fixing such an equivalence, we have a bijection between the components of $L$ and the components of $L'$, and we may consider, for any sublink of $L$, the corresponding sublink of $L'$.) Since $L_1$ is non-splittable, the corresponding sublink of  $L'$ is a sublink of $L'_{\sigma(1)}$ for some $\sigma(1) \in \{1, \dots, n'\}$. Since $L'_{\sigma(1)}$ is non-splittable, the corresponding sublink of $L$ is a sublink of $L_1$.  Thus $L_1$ and $L'_{\sigma(1)}$ are equivalent to each other via the equivalence between $L$ and $L'$. Continuing by the same reasoning, we see that
$n=n'$ and there is a permutation $\sigma$ of $\{1, \dots, n\}$ such that $L_i$ and $L'_{\sigma(i)}$ are equivalent for $i=2, \dots, n$.
\end{proof}

Let $D$ be a (classical, virtual or twisted) link diagram.   A {\it virtual path} on $D$ is an arc $\alpha$ on $D$ from a regular point to another such that it is disjoint from classical crossings and bars of $D$. 
Here a {\it  regular point} of $D$ means a point of $D$ away from classical crossings, virtual crossings and bars. 
A virtual path may pass through virtual crossings.  An arc on $D$ away from classical crossings, virtual crossings and bars is a virtual path by definition. 

Let $D$ and $D'$ be (classical, virtual or twisted) link diagrams.  If there exist virtual paths $\alpha$ and $\alpha'$ of $D$ and $D'$ respectively such that $D \setminus \alpha = D' \setminus \alpha'$,   then we say that $D'$ is obtained from $D$ by a {\it detour move} replacing $\alpha$ with $\alpha'$.  

\begin{lem}[The Detour Lemma, cf. \cite{rbor, rkk1, rkauD}]\label{detourlemma}
Let $D$ and $D'$ be virtual (or twisted, respectively) link diagrams. 
Suppose that $D'$ is obtained from $D$ by a detour move replacing a virtual path $\alpha$ with a virtual path $\alpha'$.  
Then $D$ is transformed into $D'$ by using the moves {\rm V1}, \dots, {\rm V4} (or {\rm V1}, \dots, {\rm V4} and {\rm T1}, respectively) up to isotopy of $\mathbb{R}^2$.  
In particular, $D$ and $D'$ are equivalent as virtual (or twisted, respectively) link diagrams. 
\end{lem}

\begin{proof} 
First consider a case that $D$ and $D'$ are virtual link diagrams.  There is a homotopy of arcs in $\mathbb{R}^2$, say $\alpha_t$ $(t \in [0,1])$,  with $\alpha_0 = \alpha$ and  $\alpha_1= \alpha'$ keeping the endpoints fixed.  Changing the homotopy slightly, we may assume that $\alpha_t$ are immersed arcs for all but a finite number of exceptional $t$'s.  For each regular value $t$, 
$D_t := (D \setminus \alpha ) \cup \alpha_t$ is regarded as a virtual link diagram 
by taking all crossings involving $\alpha_t$ to be virtual crossings. 
When $t$ changes continuously away from the exceptional values, the diagram $D_t$ changes by an isotopy of $\mathbb{R}^2$.  
 Before and after each exceptional value $t$, $D_t$ changes as in {\rm V1}, \dots, {\rm V4} or {\rm V4$'$}, where {\rm V4$'$} is the move obtained from V4 by switching over/under information of the crossing. 
(The arc in the move V1 should be a part of $\alpha_t$.  For V2, both or one of the two arcs are a part of $\alpha_t$.  For V3, all three  arcs, two of them or one of them are  a part of $\alpha_t$.  For V4 or V4$'$, the arc passing through the two virtual crossings is a part of $\alpha_t$.  {\rm V4} and {\rm V4$'$} occur only when the homotopy $\alpha_t$ passes over a classical crossing of $D$.) 
It is easily verified that {\rm V4$'$} is realized as a combination  of the moves {\rm V2}, {\rm V4} and {\rm V2}.  Thus we see that $D$ is transformed into $D'$ by using  {\rm V1}, \dots, {\rm V4}.  
In the case that $D$ and $D'$ are twisted link diagrams, it is proved by a similar argument.  When the homotopy $\alpha_t$ passes over a bar, we need the move T1.  Thus we see that $D$ is transformed into $D'$ by using the moves {\rm V1}, \dots,  {\rm V4} and {\rm T1}.  \end{proof}

The above lemma can be applied to a family of virtual paths.  

 A  {\it virtual path family} on a (classical, virtual or twisted) link diagram $D$ is a family of virtual paths on $D$ each of whose intersections, if it exists, is a virtual crossing.   
 
Let $D$ and $D'$ be (classical, virtual or twisted) link diagrams.  
Let $\alpha_1, \dots, \alpha_m$ be a virtual path family on $D$ and 
let $\alpha'_1, \dots, \alpha'_m$ be a virtual path family  on $D'$.  
We say that $D'$ is obtain from $D$ by {\it detour moves} by replacing  $\alpha_1, \dots, \alpha_m$ with $\alpha'_1, \dots, \alpha'_m$ if $ D \setminus (\cup_{k=1}^m \alpha_k)  = D' \setminus (\cup_{k=1}^m \alpha'_k) $ 
and for each $k =1, \dots, m$, $\partial \alpha_k = \partial \alpha'_k$.  

\begin{lem}[A Generalized Version of the Detour Lemma]\label{generaldetourlemma}
Let $D$ and $D'$ be virtual (or twisted, respectively) link diagrams. 
Suppose that $D'$ is obtained from $D$ by detour moves by replacing  $\alpha_1, \dots, \alpha_m$ with $\alpha'_1, \dots, \alpha'_m$.   
Then $D$ is transformed into $D'$ by using the moves {\rm V1}, \dots, {\rm V4} (or {\rm V1}, \dots {\rm V4} and {\rm T1}, respectively) up to isotopy of $\mathbb{R}^2$.  
In particular, $D$ and $D'$ are equivalent as virtual (or twisted, respectively) link diagrams. 
\end{lem}

\begin{proof} 
If necessary, modifying $\alpha_1, \dots, \alpha_m$ slightly, we may assume that 
for each $k =1, \dots, m$, $\alpha_k \cap \alpha'_j$ with any $j \neq k$ is empty or consists of some transverse double points.   
For each $k =1, \dots, m$,  let $D^k: = (D \setminus \cup_{i=1}^k \alpha_i) \cup (\cup_{i=1}^k \alpha'_i)$, which is regarded as a virtual (or twisted) link diagram by taking all crossings of $D^k$ involving $\cup_{i=1}^k \alpha'_i$ to be virtual crossings. 
Put $D^0= D$, and note that $D^m = D'$. 
Then $D^{k}$ is obtained from $D^{k-1}$ by a detour move replacing $\alpha_k$ with $\alpha'_k$.  
Applying Lemma~\ref{detourlemma} inductively, we obtain the result.    
\end{proof}

Let T3$'$ be the move obtained from the move T3 (Figure~\ref{fgrmoves}) 
by switching over/under information of the classical crossing.   

\begin{lem}\label{t3lemma}
The move {\rm T3$'$} is realized by the moves {\rm V1}, \dots, {\rm V4} and {\rm T3}.   
\end{lem}

\begin{proof} 
In Figure~\ref{fig:fgt}, (1) $\Rightarrow$ (2) is the  move T3, (2) $\Rightarrow$ (3) is an isotopic deformation rotating the classical crossing, and (3) $\Rightarrow $ (4) is a transformation by detour moves, where we consider $4$ virtual paths, say $\alpha_1, \dots, \alpha_4$ (or $\alpha'_1, \dots, \alpha'_4$), 
obtained from the local diagram depicted in (3) (or (4)) by removing a regular neighborhood of the classical crossing.  
Since we can consider homotopies of arcs changing $\alpha_1, \dots, \alpha_4$ to $\alpha'_1, \dots, \alpha'_4$  
without intersecting any bar, the transformation (3) $\Rightarrow $ (4) 
is realized by the moves V1, \dots, V4.  
Thus, (1) $\Rightarrow$ (4), which is the move T3$'$, is realized by the moves V1, \dots, V4 and T3. 
\end{proof}

\begin{figure}[h]
\centerline{
\includegraphics[width=7.2cm]{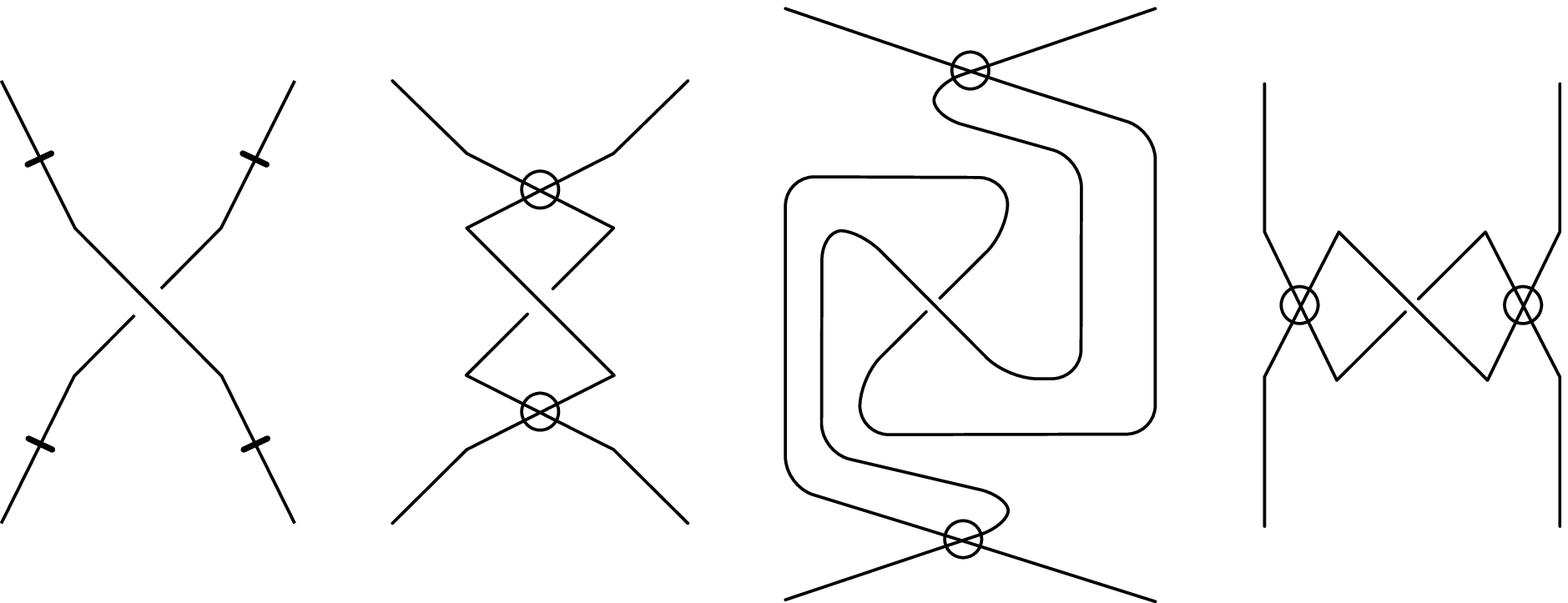}
}
\setlength{\unitlength}{1mm}
\begin{picture}(5,0)(0,0)
\put(-30,0){(1)}
\put(-10,0){(2)}
\put(10,0){(3)}
\put(30,0){(4)}
\end{picture}
\vspace{-0.2cm}
\caption{}\label{fig:fgt}
\end{figure}

Here we give a  proof of Proposition~\ref{prop1}. 
% An alternative proof,  using Gauss chord diagrams, is given in Section~\ref{sect:Gauss}.

\begin{proof}[Proof of Proposition~\ref{prop1}]
We show that for any twisted link diagram $D$, the diagrams $D$ and $s(D)$ are equivalent as twisted link diagrams.
By an ambient isotopy of $\R^2$, we may assume that $D$ lies in the half plane $\{x <0\}$ of the $xy$-plane,
and that it lies in general position with respect to the $y$-component.
By slicing along finitely many horizontal lines, $D$ has a decomposition into pieces of the types depicted in Figure~\ref{fig:fgpa}:
(i) there is a maximal point, (ii) there is a minimal point, (iii) and (iv) there is a classical crossing,
(v) there is a virtual crossing, (vi) there is a bar.
We call these pieces {\it standard pieces} and denote them
by $M_{a,b}$, $m_{a,b}$, $X^+_{a,b}$, $X^-_{a,b}$, $V_{a,b}$ and $T_{a,b}$, respectively, where $a$ (or  $b$, respectively)  is the number of vertical arcs appearing on the left (or right, respectively) of the event: a maximal point, a minimal point, a classical crossing, a virtual crossing or a bar.

\begin{figure}[h]
\centerline{
\includegraphics[width=9.0cm]{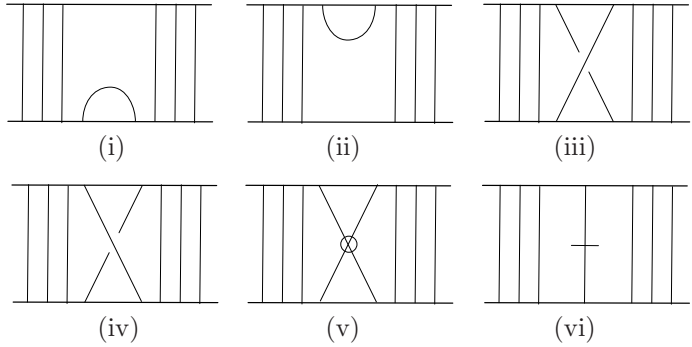}
}
\setlength{\unitlength}{1mm}
\begin{picture}(90,1)(-2,0)
\put(10,24){(i)}
\put(40,24){(ii)}
\put(70,24){(iii)}
\put(10,0){(iv)}
\put(40,0){(v)}
\put(70,0){(vi)}
\end{picture}
\caption{Standard pieces: $M_{a,b}$, $m_{a,b}$, $X^+_{a,b}$, $X^-_{a,b}$, $V_{a,b}$, and $T_{a,b}$}\label{fig:fgpa}
\end{figure}

\begin{figure}[h]
\centerline{
\includegraphics[width=3.5cm]{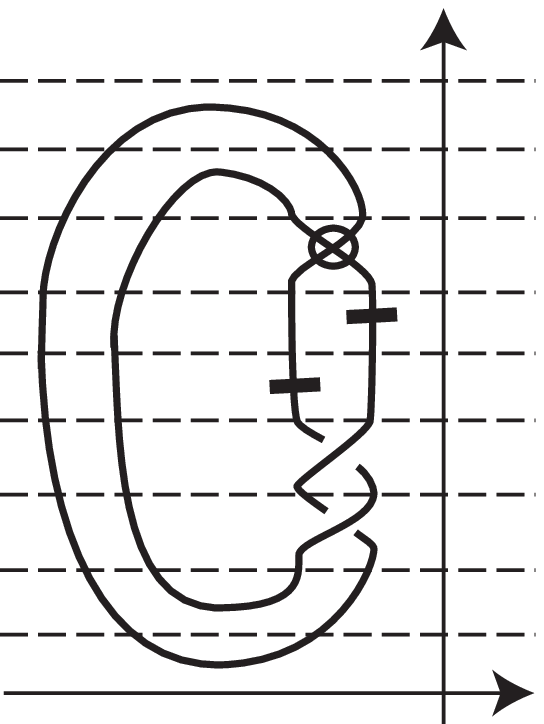}
}
%\setlength{\unitlength}{1mm}
%\begin{picture}(70,1)(-2,0)
%\put(0,24){(i)}
%\end{picture}
\vspace{-0.2cm}
\caption{}\label{fig:fgh}
\end{figure}

For $k \in \Z$, we denote by $\ell_k$ the horizontal line determined by the equality $y=k$,
and denote by $C_k$ the region of $\R^2$ determined by the inequalities
$k-1 \leq y  \leq k$.  We call $C_k$ the $k$th {\it chamber}.

Let $m$ be the total number of maximal points, minimal points, classical crossings, virtual crossings and bars of $D$.
Modifying $D$ by an isotopy of $\R^2$, we may assume that $D$ lies in $\cup_{k=1}^{m} C_k$,  and for each $k =1, \dots, m, $ the restriction of $D$ to $C_k$ is a standard piece.
For example, for the diagram in Figure~\ref{fig:fgh}, $m=9$ and $D \cap C_k,$  $k=1, \dots, 9,$ is
$m_{0,0}$, $m_{1,1}$, $X^+_{2,0}$, $X^+_{2,0}$, $T_{2,1}$, $T_{3,0}$, $V_{2,0}$, $M_{1,1}$ or $M_{0,0}$, respectively.

Let $s(D)$ be the diagram obtained from $D$ by the reflection along the $y$-axis and switching over/under information on all classical crossings.
We show that $D$ is equivalent to $s(D)$ by a sequence of extended Reidemeister moves.

Let $\delta$ be a sufficiently small positive number and, for each $k=1, \dots, m-1$, let $N(\ell_k)$ be the regular neighborhood of $\ell_k$ determined by the inequalities  $k -\delta \leq y \leq k+\delta$.
We denote by $\ell^+_k$ (or $\ell^-_k$, respectively) the horizontal line determined by the equality $y=k+\delta$ (or $y=k-\delta$, respectively).

We may assume that the intersection $D \cap N(\ell_k)$ is a collection of $d_k$ $(\geq 0)$  vertical arcs, say  $A_{k,1}, \dots, A_{k,d_k}$.  Assume that $A_{k,1}, \dots, A_{k,d_k}$ appear in this order from left to right.
Let $P_{k,j}$, $j=1, \dots, d_k$, be the intersection point of
$A_{k,j}$ and $\ell_k$.
See Figure~\ref{fig:fgpba} (Left), where $d_k=4$, and $A_{k,j}$ and $P_{k,j}$  are denoted by $A_j$ and $P_j$, respectively.

Let $P'_{k,j}$,  $j=1, \dots, d_k$,  denote the image of $P_{k,j}$ under reflection along the $y$-axis.
By virtual Reidemeister moves, we deform $A_{k,1}, \dots, A_{k,d_k}$ into arcs
$\widetilde{A}_{k,1}, \dots, \widetilde{A}_{k,d_k}$ as in Figure~\ref{fig:fgpba} such that
$\widetilde{A}_{k,j} \cap \ell_k = P'_{k,j}$ and $\partial \widetilde{A}_{k,j} = \partial A_{k,j}$ for all $j=1, \dots, d_k$.  In Figure~\ref{fig:fgpba} (Right),  $\widetilde{A}_{k,j}$ and $P'_{k,j}$ are denoted by $\widetilde{A}_j$ and $P'_j$.

Let $D_1$ be the twisted link diagram obtained from $D$ by this modification for all $k =1, \dots, m-1$.

\begin{figure}[h]
\centerline{
\includegraphics[width=12.0cm]{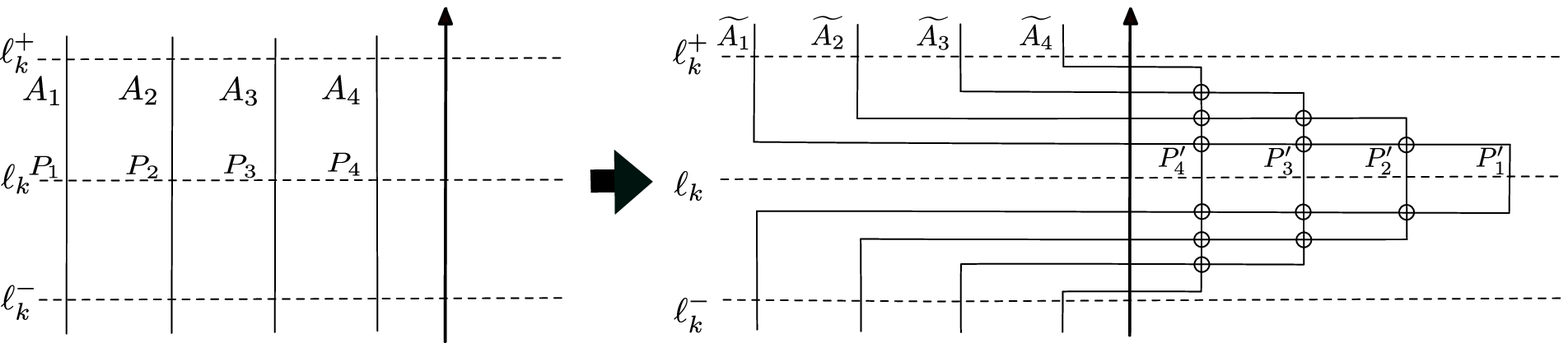}
}
\caption{}\label{fig:fgpba}
\end{figure}

We will further transform $D_1$ into a twisted link diagram $D_2$ by applying extended Reidemeister moves in each chamber $C_k$  for $k =1, \dots, m$ as below.  

For each chamber $C_k$ such that $D \cap C_k$ is of the form $M_{a,b}$, $m_{a,b}$, $V_{a,b}$ or $T_{a,b}$, we  transform $D_1 \cap C_k$ into $D_2 \cap C_k$ in the form of 
$M_{b,a}$, $m_{b,a}$, $V_{b,a}$ or $T_{b,a}$ respectively,
by extended Reidemeister moves in $C_k$.  We explain this procedure by using examples below. 

For example, consider a case that $D \cap C_k$ is of the form $M_{1,2}$. 
Figure~\ref{fig:fgpbae} shows $D \cap C_k$, $D_1 \cap C_k$ and $D_2 \cap C_k$.  
Let $\alpha_1, \dots, \alpha_4$ be the arcs of $D_1 \cap C_k$ and 
let $\alpha'_1, \dots, \alpha'_4$ be the arcs of $D_2 \cap C_k$ with $\partial \alpha_i = \partial \alpha'_i$ for $i=1, \dots, 4$, which are virtual paths. By Lemma~\ref{generaldetourlemma}, we see that $D_1 \cap C_k$ is transformed into $D_2 \cap C_k$ by extended Reidemeister moves. (In fact, we only need virtual Reidemeister moves.)  

\begin{figure}[h]
\centerline{
\includegraphics[width=12.0cm]{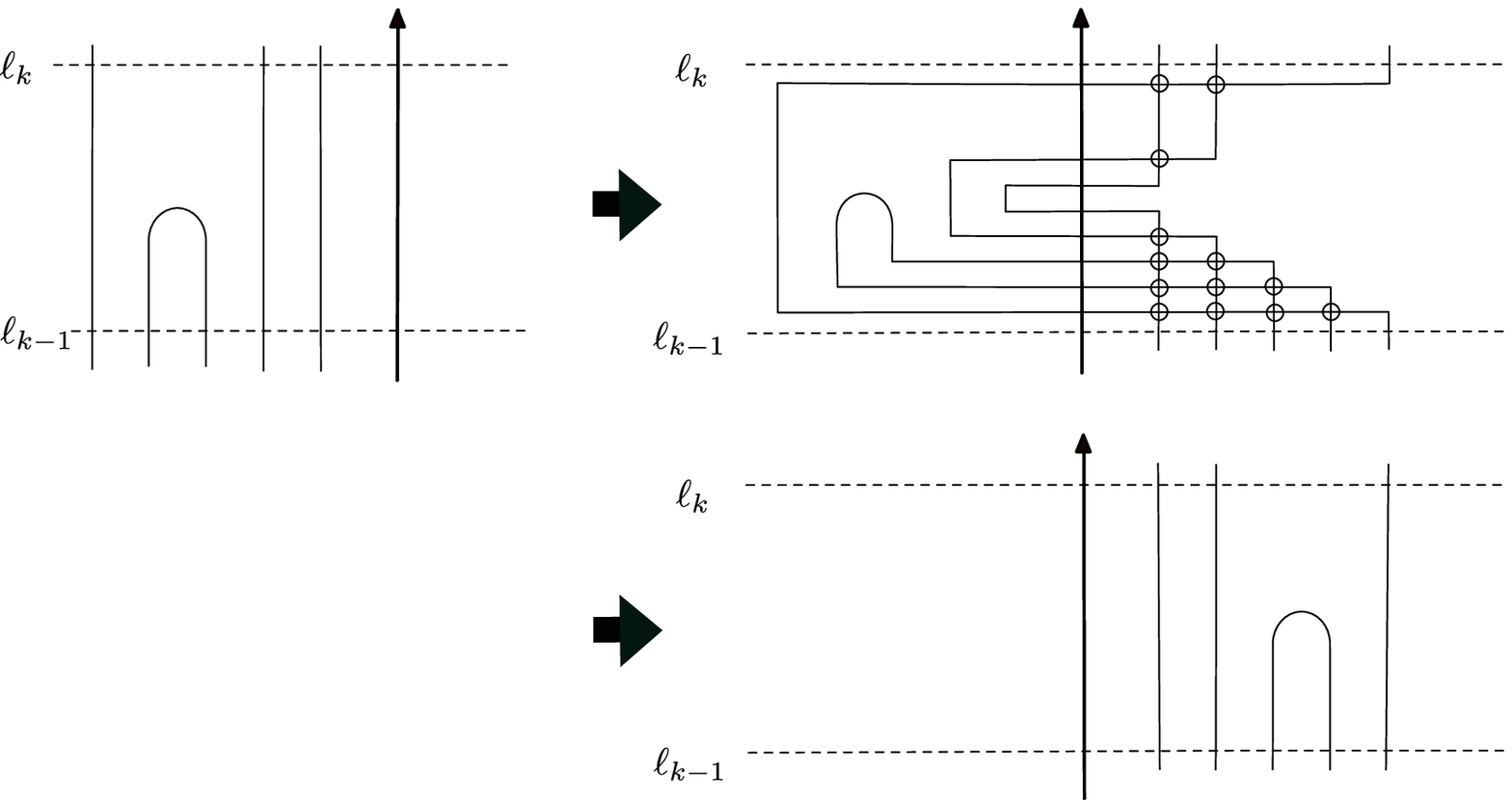}
}
\caption{}\label{fig:fgpbae}
\end{figure}

Consider a case that $D \cap C_k$ is of the form $V_{1,2}$. 
Figure~\ref{fig:fgpbac} shows  $D \cap C_k$, $D_1 \cap C_k$ and $D_2 \cap C_k$.  
Let $\alpha_1, \dots, \alpha_5$ be the arcs of $D_1 \cap C_k$ and 
let $\alpha'_1, \dots, \alpha'_5$ be the arcs of $D_2 \cap C_k$ with $\partial \alpha_i = \partial \alpha'_i$ for $i=1, \dots, 5$, which are virtual paths.   
By Lemma~\ref{generaldetourlemma}, we see that $D_1 \cap C_k$ is transformed into $D_2 \cap C_k$ by extended Reidemeister moves. (We only need virtual Reidemeister moves.)  

\begin{figure}[h]
\centerline{
\includegraphics[width=12.0cm]{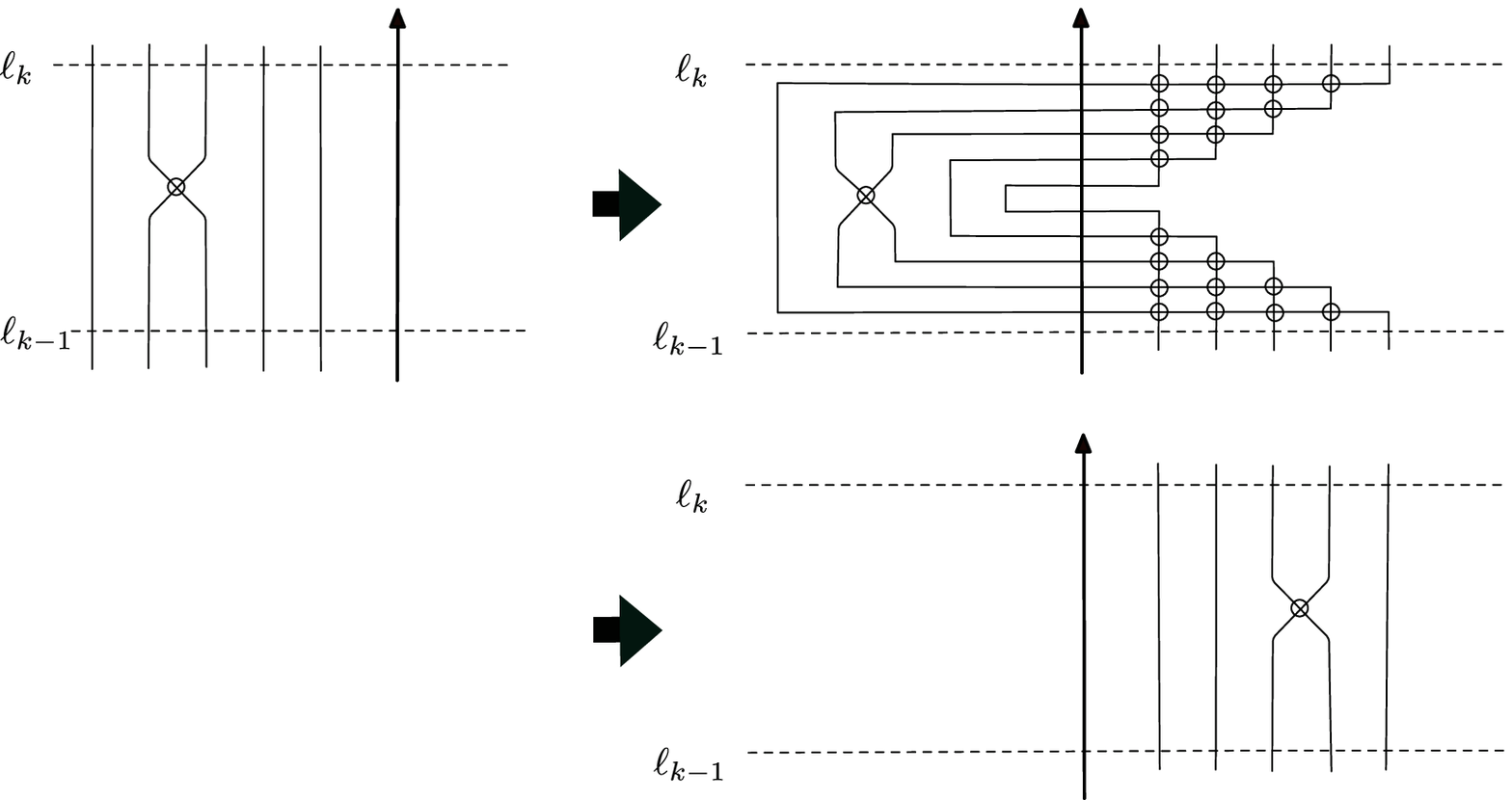}
}
\caption{}\label{fig:fgpbac}
\end{figure}

Consider a case that $D \cap C_k$ is of the form $T_{1,3}$. 
Figure~\ref{fig:fgpbad} shows  $D \cap C_k$, $D_1 \cap C_k$ and $D_2 \cap C_k$.   
Let $N$ be a regular neighborhood of the bar of $D_2 \cap C_k$ in the chamber $C_k$, and let ${\rm int}N$ denote the interior of $N$. The restriction of $D_2 \cap C_k$ to $N$ is a short vertical arc with the bar.   The restriction of $D_2 \cap C_k$ to $C_k \setminus {\rm int} N$ consists of $6$ arcs, which are virtual paths on $D_2$, say $\alpha'_1, \dots, \alpha'_6$. 
 Deform $D_1 \cap C_k$ by an isotopy of $C_k$ keeping $\partial C_k$  pointwise  fixed 
  and moving the bar of $D_1 \cap C_k$ into $N$ 
so that the new diagram, say $(D_1 \cap C_k)^\ast$, satisfies that $(D_1 \cap C_k)^\ast \cap N = D_2 \cap N$.  
The restriction of $(D_1 \cap C_k)^\ast$ to $C_k \setminus {\rm int} N$ consists of $6$ virtual paths on $(D_1 \cap C_k)^\ast$, say $\alpha_1, \dots, \alpha_6$.  Here we may assume that $\partial \alpha_i = \partial \alpha'_i$ for $i=1, \dots, 6$.  
Then $D_2 \cap C_k$ is obtained from $(D_1 \cap C_k)^\ast$ by detour moves replacing 
$\alpha_1, \dots, \alpha_6$ with $\alpha'_1, \dots, \alpha'_6$.   
By Lemma~\ref{generaldetourlemma}, we see that $(D_1 \cap C_k)^\ast$ is transformed into $D_2 \cap C_k$ by extended Reidemeister moves. Hence $D_1 \cap C_k$ is transformed into $D_2 \cap C_k$ by extended Reidemeister moves.  

\begin{figure}[h]
\centerline{
\includegraphics[width=12.0cm]{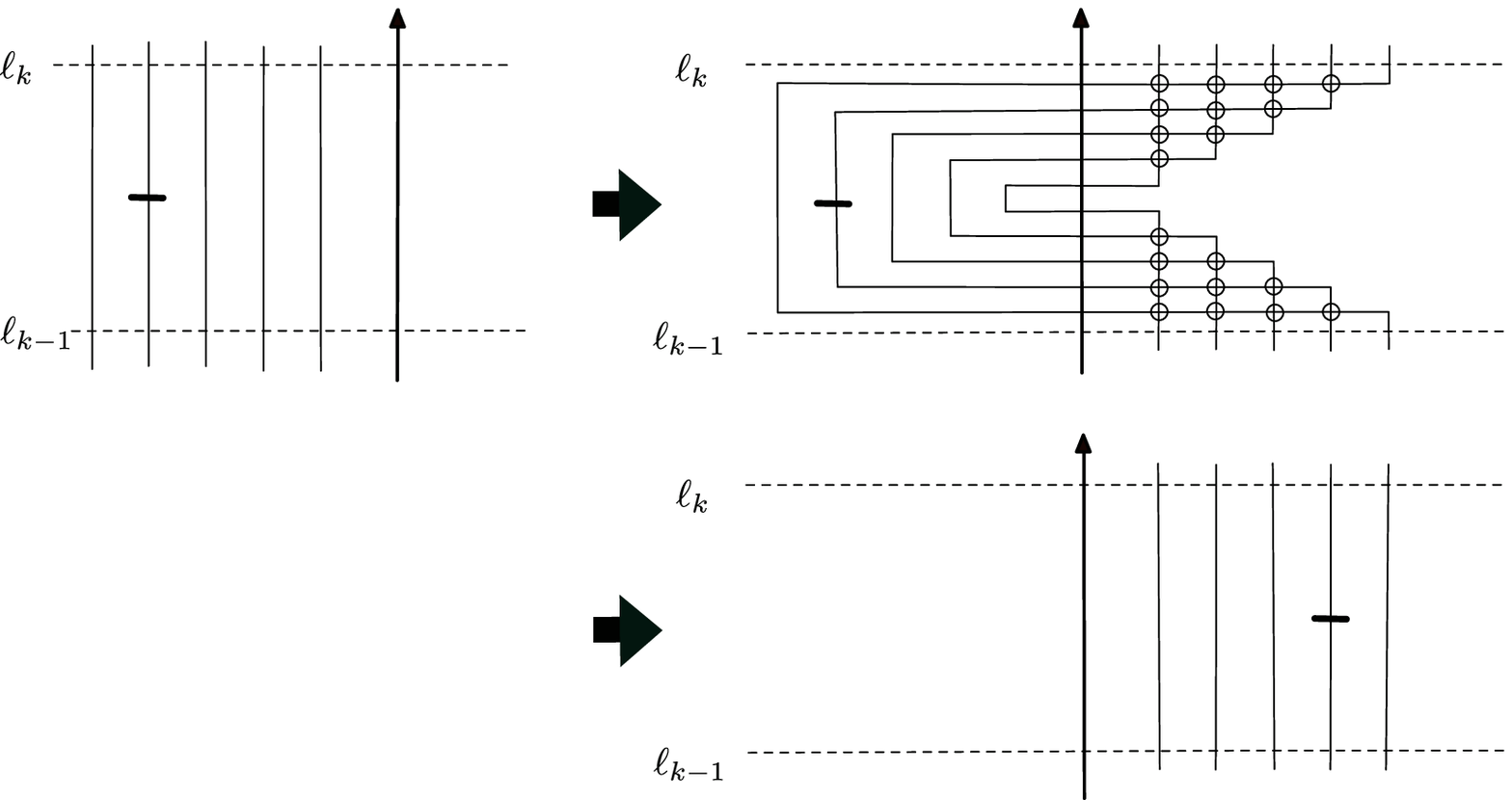}
}
\caption{}\label{fig:fgpbad}
\end{figure}

Apply this procedure for every chamber $C_k$ such that $D \cap C_k$ is of the form $M_{a,b}$, $m_{a,b}$, $V_{a,b}$ or $T_{a,b}$.  Then we obtain 
$D_2 \cap C_k$ of the form $M_{b,a}$, $m_{b,a}$, $V_{b,a}$ or $T_{b,a}$.

For each chamber $C_k$ such that $D \cap C_k$ is of the form $X^{\pm}_{a,b}$, we transform $D_1 \cap C_k$ into 
$D_2 \cap C_k$ which is 
the composition $V_{b,a} X^{\pm}_{b,a} V_{b,a}$
by extended Reidemeister moves in $C_k$ (see Figure~\ref{fig:fgpc} for $V_{b,a} X^+_{b,a} V_{b,a}$).  
We explain this procedure by using an example.  

\begin{figure}[h]
\centerline{
\includegraphics[width=4.0cm]{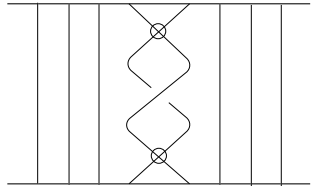}
}
\caption{$V_{b,a} X^+_{b,a} V_{b,a}$}\label{fig:fgpc}
\end{figure}

Consider a case that $D \cap C_k$ is of the form $X^+_{1,2}$. 
Figure~\ref{fig:fgpbacb} shows  $D \cap C_k$, $D_1 \cap C_k$ and $D_2 \cap C_k$.   
Let $N$ be a regular neighborhood on the classical crossing of $D_2 \cap C_k$ in $C_k$, and let ${\rm int}N$ denote the interior of $N$. The restriction of $D_2 \cap C_k$ to $C_k \setminus {\rm int} N$ consists of $7$ virtual paths on $D_2$, say $\alpha'_1, \dots, \alpha'_7$. 
Deform $D_1 \cap C_k$ by an isotopy of $C_k$ keeping $\partial C_k$ pointwise fixed and moving the classical crossing into $N$ 
  so that the new diagram, say $(D_1 \cap C_k)^\ast$, satisfies that $(D_1 \cap C_k)^\ast \cap N = D_2 \cap N$.  
The restriction of $(D_1 \cap C_k)^\ast$ to $C_k \setminus {\rm int} N$ consists of $7$ virtual paths on $(D_1 \cap C_k)^\ast$, say $\alpha_1, \dots, \alpha_7$. 
Here we may assume that $\partial \alpha_i = \partial \alpha'_i$ for $i=1, \dots, 7$.  
 Then $D_2 \cap C_k$ is obtained from $(D_1 \cap C_k)^\ast$ by detour moves replacing 
$\alpha_1, \dots, \alpha_7$ with $\alpha'_1, \dots, \alpha'_7$.   
By Lemma~\ref{generaldetourlemma}, we see that $(D_1 \cap C_k)^\ast$ is transformed into $D_2 \cap C_k$ by extended Reidemeister moves. Hence $D_1 \cap C_k$ is transformed into $D_2 \cap C_k$ by extended Reidemeister moves.  

\begin{figure}[h]
\centerline{
\includegraphics[width=12.0cm]{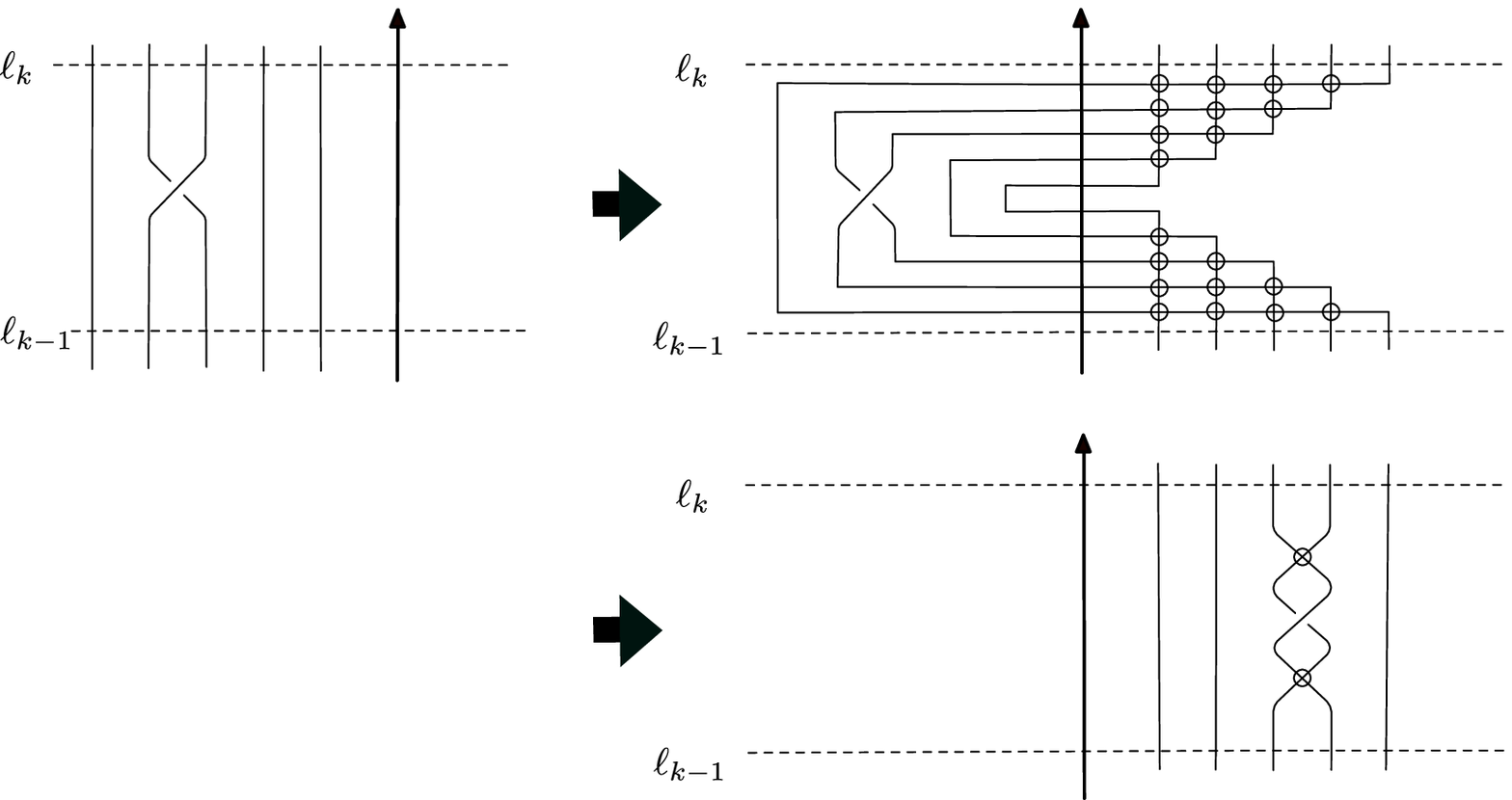}
}
\caption{}\label{fig:fgpbacb}
\end{figure}

Apply this procedure for every chamber $C_k$ such that $D \cap C_k$ is of the form $X^{\pm}_{a,b}$.  

Now the diagram $D_2$ obtained this way is equivalent to $D_1$  as a twisted link diagram, 
and hence it is equivalent to $D$.

We may assume that for each $k =1, \dots, m-1$, $D_2 \cap N(\ell_k)$ is the union of vertical arcs
$A'_{k,1}, \dots, A'_{k,d_k}$, where $A'_{k,j}$ is the image of $A_{k,j}$ under reflection along the $y$-axis.
See Figure~\ref{fig:fgpd} (Left), where $d_k = 4$, and $A'_{k,j}$ are denoted by $A'_j$.

\begin{figure}[h]
\centerline{
\includegraphics[width=10.0cm]{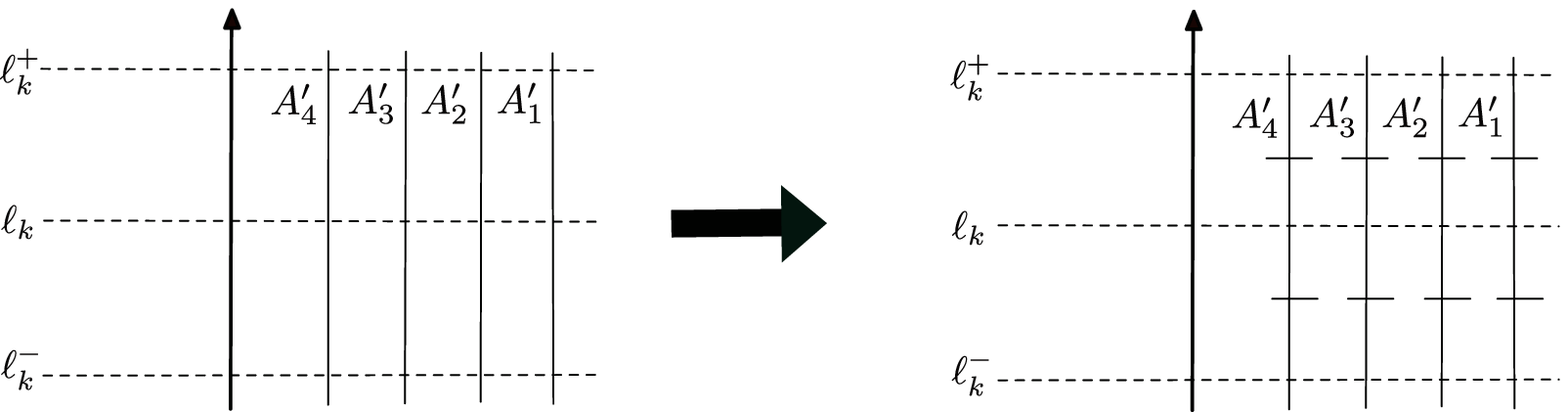}
}
\caption{}\label{fig:fgpd}
\end{figure}

Let $D_3$ be the diagram obtained from $D_2$ by adding a pair of bars on each $A'_{k,j}$,  for $k=1, \dots, m-1$ and $j=1, \dots, d_k$,  such that one of the pair lies in  $N(\ell_k) \cap C_k$ and the other lies in $N(\ell_k) \cap C_{k+1}$.  See Figure~\ref{fig:fgpd}.
In every chamber $C_k$ such that
$D_2 \cap C_k$ is of the form $M_{b,a}$, $m_{b,a}$, $V_{b,a}$ or $T_{b,a}$, the bars of $D_3 \cap C_k$ contained 
in $N(\ell_{k-1}) \cap C_k$ 
and those in $N(\ell_k) \cap C_k$ can be removed by using {\rm T2} (and {\rm T1} for $V_{b,a}$, and an isotopy for $T_{b,a}$).     
In every chamber $C_k$ such that
$D_2 \cap C_k$ is of the form $V_{b,a} X^{\pm}_{b,a} V_{b,a}$, we see that $D_3 \cap C_k$ 
can be transformed into $X^{\pm}_{b,a}$ by applying {\rm T3} or {\rm T3$'$}, followed by {\rm T2}.   Here {\rm T3$'$} is the move obtained from {\rm T3} by switching over/under information of the classical crossing, which is realized by {\rm V1}, \dots, {\rm V4} and {\rm T3} (Lemma~\ref{t3lemma}). 

Now we obtain $s(D)$, and we have shown that $D$ is equivalent to $s(D)$ as a twisted link diagram. 
\end{proof}

Given a twisted link diagram $D$, \cite{rkk7} describes a method for constructing a virtual link diagram $\widetilde{D}$, called the {\it double covering diagram} of $D$, and the following result is obtained.

\begin{thm}[\cite{rkk7}]\label{thm5}
Let $D$ and $D'$ be twisted link diagrams,  and let $\widetilde{D}$ and $\widetilde{D'}$ be double covering diagrams of  $D$ and $D'$, respectively.
If $D$ and $D'$ are equivalent as twisted link diagrams, then
$\widetilde{D}$ and $\widetilde{D'}$ are equivalent as virtual link diagrams.
\end{thm}
Therefore, for a twisted link $L$ represented by a diagram $D$, we may define the {\it double covering} $\widetilde{L}$ of $L$ to be the virtual link represented by $\widetilde{D}$, and there is a  map
$$ \{\text{twisted links}\} \rightarrow \{\text{virtual links}\}, \quad L \mapsto  \widetilde{L},   $$
called {\it the double covering}.  When $D$ is a virtual link diagram, it follows from the construction in \cite{rkk7}  that
the double covering diagram $\widetilde{D}$ is precisely $D \sqcup s(D)$.  Thus, for a virtual link $L$, the double covering $\widetilde{L}$ is $L \sqcup s(L)$.

Theorem~\ref{thm5} is used in the following proof of our main theorem.

\begin{proof}[Proof of Theorem~\ref{thm3}]
We first prove sufficiency.
Let $L$ and $L'$ be virtual links  that   are $s$-congruent.
Then there exist  maximal  split decompositions
$L = L_1 \sqcup \dots \sqcup L_n$ and $L' = L'_1 \sqcup \dots \sqcup L'_n$
such that for each $i=1, \dots, n$, $L'_i$ is equivalent to $L_i$ or $s(L_i)$ as a virtual link.
By Corollary~\ref{cor:vt}, $L'_i$ is equivalent to $L_i$ as a twisted link.
Thus $L'$ is equivalent to $L$ as a twisted link.

We next prove necessity. Let $L$ and $L'$ be virtual links  that   are equivalent  as twisted links. 
Since $L$ is a virtual link, the double covering $\widetilde{L}$ is the split union
$L \sqcup s(L)$.  Similarly, the double covering $\widetilde{L'}$ of $L'$ is the split union
$L' \sqcup s(L')$.  By Theorem~\ref{thm5}, $\widetilde{L} = L \sqcup s(L)$ is equivalent to
$\widetilde{L'} = L' \sqcup s(L')$ as a virtual link.

Let $L= L_1 \sqcup \dots \sqcup L_n$ and $L'= L'_1 \sqcup \dots \sqcup L'_{n'}$ be  maximal  split decompositions.
Then $L_1 \sqcup \dots \sqcup L_n \sqcup s(L_1) \sqcup \dots \sqcup s(L_n)$
is a  maximal  split decomposition of $\widetilde{L}$ and
$L'_1 \sqcup \dots \sqcup L'_{n'} \sqcup s(L'_1) \sqcup \dots \sqcup s(L'_{n'})$ is a  maximal  split decomposition of $\widetilde{L'}$.  By the uniqueness of a  maximal  split decomposition (Lemma~\ref{uniquemiximamdecomp}), we see that
$L$ and $L'$ are $s$-congruent.
\end{proof}

\bibliographystyle{plain}

\begin{thebibliography}{88}
\bibitem{rbor}
M.~O.~Bourgoin,
{\it Twisted link theory\/},
Algebr. Geom. Topol. {8} (2008), 1249--1279.


\bibitem{rCKS}
J.~S.~Carter, S.~Kamada and M.~Saito,
{\it Stable equivalence of knots on surfaces and virtual knot cobordisms\/},
J. Knot Theory Ramifications  {11} (2002), 311--322.

\bibitem{rDye}
H.~A.~Dye, {\it An invitation to knot theory,  Virtual and classical}, CRC Press, Boca Raton, FL, 2016.

%\bibitem{rFRS}
%R.~Fenn, C.~Rourke and B.~Sanderson,
%{\it The rack space\/},
%Trans. Amer. Math. Soc. {359} (2007),  701--740.

\bibitem{rGPV}
M.~Goussarov, M.~Polyak and O.~Viro,
{\it Finite-type invariants of classical and virtual knots\/},
Topology {39} (2000), no. 5, 1045--1068.


%\bibitem{rJKS}
%F. Jaeger, L. H. Kauffman, H. Saleur,
%{\it The Conway polynomial in $R\sp 3$ and in thickened surfaces: a new determinant formulation\/},
%Combin. Theory Ser. B  61 (1994), no. 2, 237--259.

%\bibitem{rkn0}
%N. Kamada,
%{\it On the Jones polynomials of checkerboard colorable virtual knots},
%Osaka J Math., 39 (2002), 325--333. %Osaka Journal of Mathematics,

%\bibitem{rkn1}
%N.~Kamada,
%{\it On twisted knots\/},
%Contemp. Math.Knot Theory and Its Applications, {670} (2017) 328-341.

\bibitem{rkk1}
N.~Kamada and S.~Kamada,
{\it Abstract link diagrams and virtual knots\/},
J. Knot Theory Ramifications {9} (2000), 93--106.

\bibitem{rkk7}
N.~Kamada and S.~Kamada,
{\it Double coverings of twisted links\/},
J. Knot Theory Ramifications {25} (2016), 1641011 (22 pages).
arXiv:1510.03001v3.

%\bibitem{rkns}
%N. Kamada, S. Nakabo and S. Satoh,
%{\it  A virtualized skein relation for Jones polynomials},
%Illinois~J.~Math. {46} (2002),  467--475. %Illinois Journal of Mathematics

\bibitem{rkauD}
L.~H.~Kauffman,
{\it Virtual knot theory\/},
European~J.~Combin. {20} (1999), 663--690.

%\bibitem{rkauE}
%L.~H.~Kauffman,
%{\it A self-linking invariant of virtual knots\/}. Fund. Math., 184 (2004),  135--158.

\bibitem{rKup}
G.~Kuperberg,
{\it What is a virtual link?\/},
Algebr. Geom. Topol. {3} (2003), 587--591.

\bibitem{rMant}
V.~I.~Manturov and D.~P.~Ilyutko,
{\it Virtual knots: The state of the art},
Series on Knots and Everything, 51. World Scientific Publishing Co. Pte. Ltd., Hackensack, NJ, 2013.


%\bibitem{rsakai}
%S. Tsuyoshi, {\it  Wirtinger presentations and the Kauffman bracket\/}. Kobe J. Math.,17 (2000), 83--98.
%\bibitem{rkishi}
%T. Kishino, {\it On classification of virtual links whose crossing numbers are equal to or less than 6 (in Japanese)\/}, Master Thesis, Osaka City University (2000)


%\bibitem{rMiyaB}
%Y. Miyazawa,
%{\it A Muti-variable polynomial invariant for virtual knots\/},
%J. Knot Theory Ramifications  {18} %No.5
%(2009) 625--649.


%\bibitem{rst}
%S. Satoh and  Y. Tomiyama,
%{\it On the crossing number of a virtual knot\/},
%P.  Am Math. Soc.
%PROCEEDINGS OF THE AMERICAN MATHEMATICAL SOCIETY
%{140}(2012), % Number 1, January ,
%367-376.
%S 0002-9939(2011)10917-1
%Article electronically published on May 26, 2011

\bibitem{rSaw} J.~Sawollek,
{\it On Alexander-Conway polynomials for virtual knots and links\/},
preprint (1999, arXiv:9912173(math.GT)).


\end{thebibliography}

\end{document}